\documentclass[graybox]{svmult}

\usepackage{type1cm}        % activate if the above 3 fonts are
\usepackage{multicol}        % used for the two-column index
\usepackage[bottom]{footmisc}% places footnotes at page bottom

\usepackage{newtxtext}       %
\usepackage{newtxmath}       % selects Times Roman as basic font

\makeindex             % used for the subject index
\usepackage{mathtools,amssymb,cite,dsfont}
\usepackage[shortlabels]{enumitem}
\newcommand{\R}{\mathbb R}
\newcommand{\F}{\mathcal F}
\newcommand{\FF}{\mathbb F}
\newcommand{\K}{\mathcal K}

\newcommand{\pr}{\mathbf P}
\newcommand{\eps}{\varepsilon}
\newcommand*{\abs}[1]{\left\lvert#1\right\rvert}
\newcommand*{\set}[1]{\left\{#1\right\}}
\newcommand*{\norm}[1]{\left\lVert#1\right\rVert}
\newcommand{\ex}{\mathbf E}
\newcommand{\Beta}{\mathrm B}
\newcommand{\ind}{\mathds 1}
\DeclareMathOperator{\const}{const}

\begin{document}
\title*{Stochastic differential equations driven by  additive Volterra--L\'evy   and Volterra--Gaussian noises}
\titlerunning{SDEs with additive Volterra--L\'evy noise}
\author{Giulia Di Nunno, Yuliya Mishura and Kostiantyn Ralchenko}
\institute{Giulia Di Nunno \at
	Department of Mathematics, University of Oslo\\
	P.O. Box 1053 Blindern, N-0316 Oslo, Norway\\
	\email{giulian@math.uio.no}
	\and  Yuliya Mishura \at
	Department of Probability Theory, Statistics and Actuarial Mathematics,
	Taras Shevchenko National University of Kyiv,
	64, Volodymyrs'ka St.,
	01601 Kyiv, Ukraine\\
	\email{myus@univ.kiev.ua}
	\and  Kostiantyn Ralchenko \at
	Department of Probability Theory, Statistics and Actuarial Mathematics,
	Taras Shevchenko National University of Kyiv,
	64, Volodymyrs'ka St.,
	01601 Kyiv, Ukraine\\
	\email{k.ralchenko@gmail.com}
}

\maketitle

\abstract{We study the existence and uniqueness of solutions to stochastic differential equations with Volterra processes driven by L\'evy noise. For this purpose, we study in detail smoothness properties of these processes. Special attention is given to two kinds of Volterra--Gaussian processes that generalize the compact interval representation of fractional Brownian motion and to stochastic equations with such processes.}


\section{Introduction}
The main object that is studied in the present paper  are stochastic differential equations with additive noise, admitting the form
\begin{equation}\label{main.object}
dX_t = u(X_t)dt+dY_t,\; t\ge 0,\; X|_{t=0}=X_0\in \R,
\end{equation}
where $u\colon\R\rightarrow \R$ is a measurable function, and $Y=\{Y_t, t\ge 0\}$ is a Volterra--L\'evy process. Equations of the form \eqref{main.object}, with different coefficients and different noises, were the subject of long and careful considerations.
Namely, the most popular case is the Langevin equation, where $u(x)=ax$, $x\in\R$, with some coefficient $a\neq 0$, and a Wiener process as a noise. Such process is called the Ornstein--Uhlenbeck process, or the Vasicek process, and it serves as mathematical model in many areas of science.
Initially the equation \eqref{main.object} was proposed as a model for velocity of particles in the theory of the Brownian motion in \cite{langevin1908}, then the corresponding mathematical theory was developed in \cite{uhlenbeck1930,wang1945}, see, e.\,g.\ the book \cite{VanKampen} for applications of the Ornstein--Uhlenbeck process in physics.
Since the seminal paper by Vasicek \cite{vasicek1977}, the Ornstein--Uhlenbeck process has become a very popular model in mathematical finance, see e.\,g.\ \cite{follmer1993,gibson1990,Mishura2015lmj,Mishura2015om,schobel1999, stein1991,SuWang,wiggins1987}.

A Volterra--L\'evy process has the form
$Y_t=\int_0^t g(t,s)\,dZ_s$,
where $g(t,s)$ is a given deterministic Volterra-type kernel, and $Z$ is a L\'evy process.
The conditions on $g$ and $Z$ supplying the existence of Volterra--L\'evy processes were studied in \cite{DiNunno16} together with a theory of pathwise stochastic integration with respect to such processes.
Some approximations and first numerical results can be found in \cite{New1}.
The goal of the present paper is to study stochastic differential equations with additive noise represented by a Volterra--L\'evy process.

We start with investigation of continuity and H\"older properties of Volterra--L\'evy processes.
In order to apply the Kolmogorov--Chentsov theorem, we establish moment upper bounds for increments of these processes.
In particular, we study in detail the case when the kernel $g$ satisfies certain power restrictions. Two examples of such kernels are considered, namely, the Molchan--Golosov kernel, which arises in the compact interval representation of fractional Brownian motion, and a sub-fractional kernel, which corresponds to sub-fractional Brownian motion.
For both kernels, it holds that sample paths of the corresponding Volterra--L\'evy processes satisfy H\"older condition up to order $H-\frac12$, where $H$ denotes the Hurst index. However, in the particular case of Gaussian $Z$, one has H\"older continuity up to order~$H$.
This agrees with the theory of fractional Brownian motion and with the paper \cite{Tikanmaki}, where the authors study the case, when $g(t,s)$ is the Molchan--Golosov kernel and $Z$ is a L\'evy process without Gaussian component.

Special attention in the paper is given to Volterra--Gaussian processes that arise in the case when L\'evy process $Z$ is a Brownian motion.
We investigate two types of kernels that generalize the Molchan--Golosov kernel of fractional Brownian motion. One of these kernels corresponds to fractional Brownian motion with Hurst index $H>\frac12$.
It was introduced in \cite{MSS}, where conditions for its existence and H\"older continuity were investigated.
Also, in \cite{MSS} the inverse representation of underlying Wiener process via Volterra--Gaussian process was studied. This study was based on the properties of Sonine pairs.
In the present paper we introduce also another type of Volterra--Gaussian process that extends fractional Brownian motion with $H<\frac12$.
We study smoothness of this process.
We also derive the inverse operators for both types of Volterra--Gaussian processes in terms of generalized fractional integrals and derivatives for Sonine pairs.

Then we apply the results mentioned above for investigation of stochastic differential equations with Volterra--L\'evy processes. We start with a deterministic analog of the equation \eqref{main.object}, where the stochastic term $Y_t$ is replaced by a non-random function that is locally integrable or locally bounded. We study solvability of this equation under Lipschitz condition on the drift coefficient $u$. Then we prove that the stochastic equation \eqref{main.object} with locally Lipschitz coefficient of linear growth has a unique solution under certain conditions on the underlying Volterra process $Z$ and power restrictions on the kernel $g(t,s)$.

We also study stochastic differential equations with two kinds of Volterra--Gaussian processes. In this case we can prove solvability of the equation under weaker assumptions on the drift coefficient.
Namely, we assume sublinear growth of this coefficient and its H\"older continuity.
We generalize the results of \cite{NO}, where the noise was fractional Brownian motion, to the case of more general Volterra--Gaussian noise. We prove the existence and uniqueness of a weak solution, the pathwise uniqueness of two weak solutions and the existence and uniqueness of a strong solution.

The paper is organized as follows.
In Section~\ref{sec:2} we recall the definition of a Volterra--L\'evy processes, necessary conditions for its existence, and a priory estimates for its moments.
Section~\ref{sec:3} is devoted to H\"older properties of Volterra--L\'evy processes. As auxiliary results, we establish upper bounds for the incremental moments in general case (Subsection \ref{sec:3.1}) as well as in the case of power restrictions on the kernel (Subsection \ref{sec:3.2}).
In Subsection~\ref{sec:3.3} we apply these bounds for investigation of continuity and H\"older properties of three types of Volterra--L\'evy processes.
Two examples of appropriate kernels are given in Subsection \ref{sec:3.4}.
In Subsection \ref{sec:3.5} two kinds of Volterra--Gaussian processes are studied.
Section \ref{sec:4} is devoted to the existence and uniqueness of solution to the equation \eqref{main.object}.
The stochastic differential equations with Volterra--Gaussian processes are studied in Section~\ref{sec:5}.
In Appendix we prove some auxiliary results related to fractional calculus for Sonine pairs.

Throughout the paper, we shall use notation $C$ for various constants whose value is not important and may change from line to line and even in the same line.

\section{Brief description of  Volterra--L\'evy processes}
\label{sec:2}

We start with a L\'evy process $Z$. In order to describe it, define
\[
\tau(z):=
\begin{cases}
z, & \abs{z}\le1,\\
\frac{z}{\abs{z}}, & \abs{z}>1.
\end{cases}
\]
Then the characteristic function of $Z_t$ can be represented in the following form (see, e.\,g., \cite{sato})
\[
\ex\exp\set{i\mu Z_t}
=\exp\set{t\Psi(\mu)},
\]
where \[
\Psi(\mu)=ib\mu-\frac{a\mu^2}{2}
+\int_{\R}\left(e^{i\mu x}-1-i\mu\tau(x)\right)\pi(dx),
\]
$b\in\R$,
$a\ge0$,
$\pi$ is a L\'evy measure on $\R$, that is a $\sigma$-finite Borel measure satisfying
$$\int_{\R}\left(x^2\wedge1\right)\pi(dx)<\infty,$$
with $\pi(\set{0})=0$.
The triplet $(a,b,\pi)$ is shortly called the {\it characteristic triplet of $Z$}.
Let us fix some $T>0$ and introduce  the following Volterra--L\'evy process
\begin{equation}\label{eq:volt-proc}
Y_t=\int_0^t g(t,s)\,dZ_s,
\quad t\in[0,T],
\end{equation}
where $g(t,s)$ is a given deterministic Volterra-type kernel.
The integral in \eqref{eq:volt-proc} is understood in the sense of \cite{RR} as   the limit in probability of elementary integrals.
Its construction is described in \cite[Thm.~2.2]{DiNunno16}.
According to \cite{DiNunno16}, in order to guarantee the existence  of the process $Y$ and of its moments, we need more strict assumptions on the here called \emph{base-process} $Z$ and the kernel $g(t,s)$.
More precisely, in what follows we assume that the Volterra--L\'evy process \eqref{eq:volt-proc} has $b = 0$ (i.\,e., $Z$ is a L\'evy process without drift), the measure $\pi$ is symmetric and
one of the following conditions holds:
\begin{enumerate}[label=\textbf{(A\arabic*)}]
\item\label{(A1)}
There exists   $p\in[1,2)$ such that  $g=g(t,\cdot)\in L_p([0,t])$ for any $t\in[0,T]$;
$a=0$ and  $\int_\R\abs{x}^p\,\pi(dx)<\infty$;
\item\label{(A2)} There exists  $p\ge2$ such that  $g=g(t,\cdot)\in L_p([0,t])$ for any $t\in[0,T]$ and  $\int_{\R}\abs{x}^p\pi(dx)<\infty$.
\end{enumerate}
 Then, according to \cite[Thm.~2.2]{DiNunno16}, the integral $\int_0^t g(t,s)\,dZ_s$ exists for any $t\in[0,T]$.
Moreover, in the case when condition \ref{(A1)} holds, we have  the following a priori estimate
\begin{equation}\label{eq:est1}
\ex\abs{\int_0^tg(t,s)\,dZ_s}^p
\le C\norm{g(t,\cdot)}_{L_p([0,t])}^p \int_{\R}|x|^p\pi(dx),
\end{equation}
and   in the case when condition \ref{(A2)} holds, we have the following a priori estimate
 \begin{equation}\label{eq:est2}
\ex\abs{\int_0^tg(t,s)\,dZ_s}^p
\le C\left(a^{p/2}\norm{g(t,\cdot)}_{L_2([0,t])}^p+\norm{g(t,\cdot)}_{L_p([0,t])}^p\int_{\R}\abs{x}^p\pi(dx)\right).
\end{equation}
The constant $C$ in \eqref{eq:est1} and \eqref{eq:est2} does not depend on the function $g$. However, it may depend on $p$ and $T$.

\section{Moment  upper bounds and H\"older properties of Volterra--L\'evy processes}
\label{sec:3}
In our approach, in order to consider a Volterra--L\'evy process as a noise, we need in the smoothness properties of its trajectories. So, the present section is devoted to its H\"older properties. Obviously, these properties depend both on the properties of the kernel $g$ and the L\'evy baseprocess $Z$.

\subsection{General  upper bounds for the incremental moments}
\label{sec:3.1}
In this subsection  we establish upper bounds for $\ex\abs{Y_t-Y_s}^p$ under   assumptions \ref{(A1)} and \ref{(A2)}.

\begin{lemma}\label{l:gen_bounds}
Consider  $0\le s\le t\le T$.

Let assumption \ref{(A1)} hold. Then
\begin{equation}\label{eq:gen_ineq0}
\ex\abs{Y_t-Y_s}^p
\le C \int_{\R}|x|^p\pi(dx) \left (\int_s^t\abs{g(t,u)}^p\,du
+ \int_0^s\abs{g(t,u)-g(s,u)}^p\,du \right).
\end{equation}

Let assumption \ref{(A2)} hold. Then
\begin{multline} \label{eq:gen_ineq}
\ex\abs{Y_t-Y_s}^p
\le C \int_{\R}|x|^p\pi(dx) \left (\int_s^t\abs{g(t,u)}^p\,du
+ \int_0^s\abs{g(t,u)-g(s,u)}^p\,du \right)\\
+C a^{p/2} \left (\left (\int_s^t\abs{g(t,u)}^2\,du\right)^{p/2}
+ \left(\int_0^s\abs{g(t,u)-g(s,u)}^2\,du\right)^{p/2} \right).
\end{multline}
\end{lemma}

\begin{proof}
Note that the increment of $Y$ is given by
\begin{align*}
Y_t-Y_s&=\int_0^tg(t,u)dZ_u-\int_0^sg(s,u)dZ_u
\notag\\
&=\int_s^tg(t,u)dZ_u+\int_0^s(g(t,u)-g(s,u))dZ_u.
\end{align*}
Therefore,
\begin{equation}\label{eq:increm}
\ex\abs{Y_t-Y_s}^p \le C\left(\ex\abs{\int_s^tg(t,u)dZ_u}^p
+ \ex\abs{\int_0^s(g(t,u)-g(s,u))dZ_u}^p\right).
\end{equation}
In order to conclude the proof, it suffices to apply the bounds \eqref{eq:est1} and \eqref{eq:est2} to the integrals in the right-hand side of \eqref{eq:increm}.
\end{proof}

We remark that the H\"older continuity of paths is a central property also e.\,g.\ in the rough-paths approach to the study of stochastic (partial) differential equations.
Our results can then find application in that framework.
We refer to e.\,g.\ \cite{New3} for a study of Volterra-driven stochastic differential equations with multiplicative noise via rough-paths.
Note that, different from our work, the starting base-process is H\"older continuous.

\subsection{Incremental moments and H\"older continuity under power restrictions  on the kernel $g$}\label{sec:3.2}
As one can see from the inequalities \eqref{eq:gen_ineq0} and \eqref{eq:gen_ineq}, the incremental moments of $Y$ are bounded by some integrals containing $g$, its powers and its increments.  Now let us consider more specific class of the kernels $g$.
Assume that the function $g$ satisfies the following power restrictions with some $p\ge1$.
\begin{enumerate}[label=\textbf{(B\arabic*)}]
\item\label{(B1)}
There exist constants $\alpha\in\R$, $\beta>-\frac1p$ and $\gamma>-\frac1p$ such that for all $0<u<t\le T$,
\[
\abs{g(t,u)} \le C t^\alpha u^\beta (t-u)^\gamma.
\]
\item\label{(B2)}
There exist a constant $\delta>0$ and a function $h(t,s,u)$
\[
\abs{g(t,u)-g(s,u)} \le \abs{t-s}^\delta h(t,s,u)
\quad\text{for all }0<u<s<t\le T,
\]
and
$\displaystyle \sup_{0<s<t\le T}\int_0^s \abs{h(t,s,u)}^p\,du < \infty$.
\end{enumerate}

As we shall see further on in the examples, these conditions on the kernel are well motivated by the fractional and sub-fractional Brownian motions.
An extension of condition \ref{(B1)} is provided in Remark~\ref{rem:generalization} at the end of the next subsection.

Our   goal in this and the next subsection is to obtain an inequality of the form
\[
\ex\abs{Y_t-Y_s}^p \le C \abs{t-s}^c
\]
with some $c>0$. In particular, if we get such an inequality with $c>1$, we will be able to apply the Kolmogorov continuity theorem and to investigate H\"older properties of $Y$.
Taking into account Lemma~\ref{l:gen_bounds}, we need to estimate the integrals of the form $\int_s^t\abs{g(t,u)}^p\,du$ and $\int_0^s\abs{g(t,u)-g(s,u)}^p\,du$.
Obviously, the second integral under the assumption \ref{(B2)} satisfies the inequality
\begin{equation}\label{eq:bound-int-dif}
\int_0^s\abs{g(t,u)-g(s,u)}^p\,du \le C\abs{t-s}^{\delta p}.
\end{equation}
The study of the first integral is more delicate.
We start with the following auxiliary result.

\begin{lemma}\label{lem_1}
Let $\mu>-1$ and $\nu>-1$.
Then for all $0\le s<t\le T$,
\begin{equation}\label{eq:bound1}
\int_s^t u^\mu (t-u)^\nu \,du \le C t^\mu (t-s)^{\nu+1}.
\end{equation}
 The positive constant $C$ in \eqref{eq:bound1} may depend on $\mu$, $\nu$ and $T$.
\end{lemma}

\begin{proof}
Write
\begin{equation} \label{eq:l1-1}
\int_s^t u^\mu (t-u)^\nu \,du
= \int_s^{\frac{s+t}{2}} u^\mu (t-u)^\nu \,du
+\int_{\frac{s+t}{2}}^t u^\mu (t-u)^\nu \,du
\eqqcolon I_1 + I_2.
\end{equation}
For $s\le u\le t$, we have
\[
(t-u)^\nu = (t-u)^{\nu+1}(t-u)^{-1}
\le (t-s)^{\nu+1} (t-u)^{-1}.
\]
Therefore,
\begin{align}
I_1&\le (t-s)^{\nu+1} \int_s^{\frac{s+t}{2}} \frac{u^\mu}{t-u} \,du
= (t-s)^{\nu+1} t^{-1} \int_s^{\frac{s+t}{2}} \frac{u^\mu(t-u+u)}{t-u} \,du
\notag\\
&= (t-s)^{\nu+1} t^{-1} \int_s^{\frac{s+t}{2}} u^\mu \,du
+ (t-s)^{\nu+1} t^{-1} \int_s^{\frac{s+t}{2}} \frac{u^{\mu+1}}{t-u} \,du
\eqqcolon I_{11}+I_{12}.
\label{eq:l1-2}
\end{align}
The term $I_{11}$ can be bounded as follows:
\begin{equation}\label{eq:l1-3}
I_{11} = C (t-s)^{\nu+1} t^{-1}
\left( \left(\frac{s+t}{2}\right)^{\mu+1}
- s^{\mu+1}\right )
\le C t^{\mu} (t-s)^{\nu+1},
\end{equation}
since
$\left(\frac{s+t}{2}\right)^{\mu+1} - s^{\mu+1}
\le \left(\frac{s+t}{2}\right)^{\mu+1} \le t^{\mu+1}$.

In order to bound $I_{12}$, we use the inequality
$u^{\mu+1} \le \left(\frac{s+t}{2}\right)^{\mu+1} \le t^{\mu+1}$.
We get
\begin{align}
I_{12} &\le (t-s)^{\nu+1} t^{\mu} \int_s^{\frac{s+t}{2}} \frac{du}{t-u}
=   (t-s)^{\nu+1} t^{\mu} \left ( \log (t-s) - \log \frac{t-s}{2}\right)
\notag\\
&=t^{\mu} (t-s)^{\nu+1} \log 2
= C t^{\mu} (t-s)^{\nu+1}.
\label{eq:l1-4}
\end{align}

Consider $I_2$.
Note that for $\frac{s+t}{2}<u<t$,
\begin{align*}
u^{\mu}&\le\left(\frac{s+t}{2}\right)^{\mu}\le \left(\frac{t}{2}\right)^{\mu}
& \text{if } \mu<0,
\\
u^{\mu}&\le t^{\mu}
& \text{if } \mu\ge 0.
\end{align*}
Hence, in both cases we have the bound $u^\mu\le C t^\mu$.
Therefore,
\begin{equation}\label{eq:l1-5}
I_2 \le C t^\mu\int_{\frac{s+t}{2}}^t  (t-u)^\nu \,du
= C t^\mu \left (t-\frac{s+t}{2}\right )^{\nu+1}
= C t^\mu (t-s)^{\nu+1}.
\end{equation}
Combining \eqref{eq:l1-1}--\eqref{eq:l1-5}, we get \eqref{eq:bound1}.
\end{proof}

Lemma \ref{lem_1} allows us to obtain an upper bound for the integral $\int_s^t\abs{g(t,u)}^p\,du$.
\begin{lemma}\label{l:bound-int-g}
Assume that condition \ref{(B1)} holds with some $p\ge1$.
Then for all $0\le s<t\le T$,
\[
\int_s^t \abs{g(t,u)}^p\,du \le C(t-s)^{\kappa p + 1},
\]
where
\begin{equation}\label{eq:eps}
\kappa = \kappa(\alpha,\beta,\gamma) =
\begin{cases}
\alpha + \beta + \gamma, & \text{if } \alpha + \beta < 0,
\\
\gamma, & \text{if } \alpha + \beta \ge 0.
\end{cases}
\end{equation}
The constant $C$ may depend on $\alpha$, $\beta$, $\gamma$, $p$ and $T$.
\end{lemma}

\begin{proof}
According to condition \ref{(B1)},
\[
\int_s^t \abs{g(t,u)}^p\,du \le C t^{\alpha p} \int_s^t u^{\beta p} (t-u)^{\gamma p}\,du.
\]
Applying the upper bound \eqref{eq:bound1}, we get
\[
\int_s^t \abs{g(t,u)}^p\,du \le C t^{(\alpha + \beta) p} (t-s)^{\gamma p + 1}.
\]

If $\alpha + \beta < 0$, then $t^{(\alpha + \beta) p}\le (t-s)^{(\alpha + \beta) p}$, and we obtain the inequality
\[
\int_s^t \abs{g(t,u)}^p\,du \le C (t-s)^{(\alpha + \beta + \gamma) p + 1}.
\]

If $\alpha + \beta \ge 0$, then $t^{(\alpha + \beta) p}\le T^{(\alpha + \beta) p}$, hence,
\[
\int_s^t \abs{g(t,u)}^p\,du \le C (t-s)^{\gamma p + 1}.
\]
This concludes the proof.
\end{proof}

\subsection{Application of the upper bounds for the incremental moments to   Volterra--L\'evy processes of three types}
\label{sec:3.3}
Now, basing on Lemma \ref{l:bound-int-g}, we can better specify the upper bounds \eqref{eq:gen_ineq0} and \eqref{eq:gen_ineq} for the moments of increments of the Volterra--L\'evy process $Y$ satisfying \ref{(B1)}--\ref{(B2)}.
Also, as a consequence, we shall state its H\"{o}lder properties.
We consider three cases: 1) $Z$ is a L\'evy process without Brownian part; 2) $Z$ is a Brownian motion;
3) $Z$ is a L\'evy process of a general form.

\subsubsection{L\'evy--based process without Brownian part}
We start with the case of a L\'evy process in \eqref{eq:volt-proc} without Brownian part, that is, $a=0$.
\begin{lemma}\label{l:pure-jump}
Assume that $p\ge1$,
$a=0$, $\int_\R\abs{x}^p\,\pi(dx)<\infty$,
the conditions \ref{(B1)} and \ref{(B2)} hold with some $\alpha\in\R$, $\delta>0$, $\beta>-\frac1p$, $\gamma>-\frac1p$ and such that $\alpha+\beta+\gamma>-\frac1p$.
Then for all $0\le s<t\le T$,
\[
\ex\abs{Y_t-Y_s}^p \le C(t-s)^{\min\set{\kappa p + 1,\delta p}},
\]
where $\kappa$ is defined by \eqref{eq:eps}.
If $\kappa>0$ and $\delta>\frac1p$, then the trajectories of $Y$ are a.\,s.\ H\"older continuous up to order $\min\set{\kappa,\delta-\frac1p}$.
\end{lemma}

\begin{proof}
According to Lemma~\ref{l:gen_bounds}, we have
\[
\ex\abs{Y_t-Y_s}^p
\le C \left (\int_s^t\abs{g(t,u)}^p\,du
+ \int_0^s\abs{g(t,u)-g(s,u)}^p\,du \right).
\]
Applying Lemma~\ref{l:bound-int-g} and \eqref{eq:bound-int-dif}, we get
\begin{align*}
\ex\abs{Y_t-Y_s}^p
&\le C (t-s)^{\kappa p + 1} + C (t-s)^{\delta p}
\\
&\le C T^{\kappa p + 1} \left(\frac{t-s}{T}\right )^{\min\set{\kappa p + 1,\delta p}}
+ C T^{\delta p}\left(\frac{t-s}{T}\right )^{\min\set{\kappa p + 1,\delta p}}
\\
&\le C (t-s)^{\min\set{\kappa p + 1,\delta p}}.
\end{align*}
H\"older continuity follows from the Kolmogorov continuity theorem.
\end{proof}

\subsubsection{The Brownian case}
\begin{lemma}\label{l:Brownian}
Assume that $Z$ is a Brownian motion,
the conditions \ref{(B1)} and \ref{(B2)} hold with $p=2$, $\alpha\in\R$, $\beta>-\frac12$, $\gamma>-\frac12$ such that $\alpha+\beta+\gamma>-\frac12$.
Then for all $p\ge2$ and all $0\le s<t\le T$,
\[
\ex\abs{Y_t-Y_s}^p \le C(t-s)^{p \min\set{\kappa  + \frac12,\delta}},
\]
where $\kappa$ is defined by \eqref{eq:eps}.
If $\kappa>-\frac12$, then the trajectories of $Y$ are a.\,s.\ H\"older continuous up to order $\min\set{\kappa+\frac12,\delta}$.
\end{lemma}
\begin{proof}
In the Brownian case, \eqref{eq:gen_ineq} becomes
\[
\ex\abs{Y_t-Y_s}^p
\le C \left (\left (\int_s^t\abs{g(t,u)}^2\,du\right)^{p/2}
+ \left(\int_0^s\abs{g(t,u)-g(s,u)}^2\,du\right)^{p/2} \right).
\]
Then by Lemma~\ref{l:bound-int-g} and \eqref{eq:bound-int-dif}, we get
\[
\ex\abs{Y_t-Y_s}^p\le C (t-s)^{\frac{p}{2}(2 \kappa + 1)} + C (t-s)^{\delta p}
\le C(t-s)^{p \min\set{\kappa  + \frac12,\delta}}.
\]
By the Kolmogorov continuity theorem, if $p \min\set{\kappa  + \frac12,\delta}>1$, then the trajectories of $Y$ are a.\,s.\ H\"older up to order
$\min\set{\kappa  + \frac12,\delta}-\frac1p$. Since $p$ can be chosen arbitrarily large, we get H\"older continuity up to order $\min\set{\kappa  + \frac12,\delta}$, if $\kappa>-1/2$.
\end{proof}
\subsubsection{L\'evy--based process of a general form}
Now let us consider a L\'evy process $Z$ of a general form.
In this case we need to assume that $p\ge2$ in order to guarantee the existence of $Y$ and its moments, see \cite[Thm. 2.2]{DiNunno16}.
It turns out that under this assumption we have the same upper bound for the incremental moment as in the case $a=0$.
\begin{lemma}\label{l:general}
Assume that for some $p\ge2$ we have
$\int_\R\abs{x}^p\,\pi(dx)<\infty$ and
the conditions \ref{(B1)} and \ref{(B2)} hold with some $\alpha\in\R$, $\beta>-\frac1p$, $\gamma>-\frac1p$ such that $\alpha+\beta+\gamma>-\frac1p$.
Then for all $0\le s<t\le T$,
\[
\ex\abs{Y_t-Y_s}^p \le C(t-s)^{\min\set{\kappa p + 1,\delta p}},
\]
where $\kappa$ is defined by \eqref{eq:eps}.
If $\kappa>0$ and $\delta>\frac1p$, then the trajectories of $Y$ are a.\,s.\ H\"older continuous up to order $\min\set{\kappa,\delta-\frac1p}$.
\end{lemma}

\begin{proof}
Applying Lemma~\ref{l:gen_bounds}, Lemma~\ref{l:bound-int-g} and \eqref{eq:bound-int-dif}, we obtain
\begin{align*}
\ex\abs{Y_t-Y_s}^p
&\le C \left (\int_s^t\abs{g(t,u)}^p\,du
+ \int_0^s\abs{g(t,u)-g(s,u)}^p\,du \right.\\
&\quad + \left .\left (\int_s^t\abs{g(t,u)}^2\,du\right)^{p/2}
+ \left(\int_0^s\abs{g(t,u)-g(s,u)}^2\,du\right)^{p/2} \right)
\\
&\le C (t-s)^{\kappa p + 1} + C (t-s)^{\delta p} + C (t-s)^{\frac{p}{2}(\kappa p + 1)}
\\
&\le C (t-s)^{\min\set{\kappa p + 1,\delta p,\frac{p}{2}(\kappa p + 1)}}
= C (t-s)^{\min\set{\kappa p + 1,\delta p}}.
\end{align*}
H\"older continuity follows from the Kolmogorov continuity theorem.
\end{proof}

\begin{remark}\label{rem:generalization}
The assumption \ref{(B1)} can be replaced by the following more general condition:
\begin{enumerate}[label=\textbf{(B\arabic*$'$)}]
\item\label{(B1')}
There exist constants $\alpha_i\in\R$, $\beta_i>-\frac1p$ and $\gamma_i>-\frac1p$, $i=1,2,\dots,m$, such that for all $0<u<t\le T$,
\[
\abs{g(t,u)} \le C \sum_{i=1}^m t^{\alpha_i} u^{\beta_i} (t-u)^{\gamma_i}.
\]
\end{enumerate}
In this case the statements of Lemmas \ref{l:bound-int-g}--\ref{l:general} hold true
with $\kappa=\min\limits_{1\le i\le m} \kappa_i$,
where $\kappa_i=\kappa(\alpha_i,\beta_i,\gamma_i)$, $i=1,\dots,m$, are defined by
\eqref{eq:eps}.
Indeed, in order to proof Lemma~\ref{l:bound-int-g} under the assumption \ref{(B1')}, it suffices to apply the bound $(x_1+\dots+x_m)^p\le C\left(x_1^p+\dots+x_m^p\right)$ and follow the same reasoning as in the case of the condition \ref{(B1)}.
Other lemmas are then easily deduced from Lemma~\ref{l:bound-int-g}.
\end{remark}

\subsection{Examples of Volterra--L\'evy processes with power restrictions on the kernel}
\label{sec:3.4}

\subsubsection{The Molchan--Golosov kernel}

Let us verify the assumptions \ref{(B1)} and \ref{(B2)} for the Molchan--Golosov kernel, which is defined as
\begin{equation}\label{eq:MG-kernel}
K_H(t,s) = C_H s^{\frac12-H}\left (t^{H-\frac12}(t-s)^{H-\frac12} - (H-\tfrac12) \int_s^t u^{H-\frac32}(u-s)^{H-\frac12}\,du\right),
\end{equation}
where $H\in(0,1)$, \[
C_H = \left(\frac{2H \Gamma(H+\frac12) \Gamma(\frac32-H)}{\Gamma(2-2H)}\right)^{\frac12}.
\]
This kernel arises in the compact interval representation of the fractional Brownian motion as an integral with respect to a Wiener process $W$, see, e.\,g., \cite[Section 2.8]{Mishura_Zili}.
More precisely, the Volterra process
\begin{equation}\label{fBm}B^H_t=\int_0^t K_H(t,s)\,dW_s,\;  t\ge0\end{equation}  is a fractional Brownian motion with the Hurst parameter $H$, that is a zero mean Gaussian process with    covariance function
$$\ex B^H_tB^H_s=\frac12\left(s^{2H}+t^{2H}-|t-s|^{2H}\right).$$
Note that the precise value of $C_H$ is irrelevant in the context of our study, the following results concerning H\"older continuity of Volterra processes are valid for any $C>0$ instead of $C_H$.

Hereafter we consider the Volterra process
\begin{equation}\label{eq:Levy-fBm}
Y^H_t=\int_0^t K_H(t,s)\,dZ_s,\quad t\in[0,T],
\end{equation}
where $Z$ is a L\'evy base-process.
We recall that if $Z$ is without Gaussian component, then the process \eqref{eq:Levy-fBm} is known as \emph{fractional L\'evy process by Molchan--Golosov transformation}. It was introduced and studied in \cite{Tikanmaki}.

\begin{proposition}\label{p:fBm}
Let $H\in(0,1)$, $\eps\in(0,H)$.
\begin{enumerate}[1.]
\item
Let
$0<\int_\R x^2\,\pi(dx)<\infty$.
Then for all $0\le s<t\le T$,
\[
\ex\abs{Y^H_t-Y^H_s}^2 \le C(t-s)^{2(H-\eps)}.
\]
If $H\in(\frac12,1)$, then the trajectories of $Y^H$ are $\varkappa$-H\"older continuous for any $\varkappa\in(0,H-\frac12)$.

\item
Let $Z$ be a Brownian motion. Then
for all $p\ge2$ and all $0\le s<t\le T$,
\[
\ex\abs{Y^H_t-Y^H_s}^p \le C(t-s)^{p(H-\eps)},
\]
and the trajectories of $Y^H$ are $\varkappa$-H\"older continuous for any $\varkappa\in(0,H)$.
\end{enumerate}
\end{proposition}

\begin{proof}
We prove both statements simultaneously.
Without loss of generality,   assume that $0<\eps<\min\set{1-H,\frac12}$. Indeed, if the result of the proposition holds for some $\eps=\eps^*>0$, then it holds also for all $\eps>\eps^*$.
We consider the cases $H=\frac12$, $H>\frac12$ and $H<\frac12$ separately.

\emph{Case $H=\frac12$.}
Note that if $H=\frac12$, then $K_H\equiv \const$.
Hence, for any $p$, \ref{(B1)} and \ref{(B2)} are valid with $\alpha=\beta=\gamma=0$ and with any $\delta>0$.
If $\int_\R x^2\,\pi(dx)<\infty$, then, by Lemma~\ref{l:general},
\[
\ex\abs{Y^H_t-Y^H_s}^2 \le C(t-s)
\]
for all $0\le s<t\le T$.
If  $Z$ is a Brownian motion, then
by Lemma \ref{l:Brownian},
\[
\ex\abs{Y^H_t-Y^H_s}^p \le C(t-s)^{\frac{p}{2}}
\]
for all $0\le s<t\le T$ and $p\ge2$.
Hence, both statements of the proposition hold even for $\eps=0$ (consequently, they hold for any $\eps>0$).

\emph{Case $H\in(\frac12,1)$.}
In this case the kernel \eqref{eq:MG-kernel} can be rewritten using integration by parts in the following form:
\begin{equation}\label{eq:MG>1/2}
K_H(t,s) = C s^{\frac12-H}\int_s^t u^{H-\frac12}(u-s)^{H-\frac32}\,du.
\end{equation}
For $0<s<t\le T$, we have
\[
\abs{K_H(t,s)} \le C s^{\frac12-H}t^{H-\frac12}\int_s^t (u-s)^{H-\frac32}\,du
= C t^{H-\frac12}s^{\frac12-H}(t-s)^{H-\frac12}.
\]
Therefore the condition \ref{(B1)} holds with $\alpha=H-\frac12$, $\beta=\frac12-H$, $\gamma=H-\frac12$.

In order to verify the condition \ref{(B2)}, we need to estimate the difference $|K_H(t,u)-K_H(s,u)|$.
We have for $0<u<s<t\le T$,
\begin{align}
\abs{K_H(t,u)-K_H(s,u)} &= C u^{\frac12-H}\int_s^t z^{H-\frac12}(z-u)^{H-\frac32}\,dz
\notag\\
&\le C u^{\frac12-H}\int_s^t (z-u)^{2H-2}\,dz
+ C \int_s^t (z-u)^{H-\frac32}\,dz,
\label{eq:example-bound-dif}
\end{align}
(here we have used the inequality
$z^{H-\frac12} \le (z-u)^{H-\frac12} + u^{H-\frac12}$).
Let $\eps\in(0,1-H)$.
Then the integrals in the right-hand side of \eqref{eq:example-bound-dif} can be bounded as follows:
\begin{align*}
\int_s^t (z-u)^{2H-2}\,dz
&\le (s-u)^{H+\eps-1}\int_s^t (z-s)^{H-\eps-1}\,dz
\\
&= C (s-u)^{H+\eps-1}(t-s)^{H-\eps},
\end{align*}
\[
\int_s^t (z-u)^{H-\frac32}\,dz
\le (s-u)^{\eps-\frac12} \int_s^t (z-s)^{H-\eps-1}\,dz
= C (s-u)^{\eps-\frac12} (t-s)^{H-\eps}.
\]
Hence,
\[
\abs{K_H(t,u)-K_H(s,u)}
\le (t-s)^{H-\eps} h(s,u),
\]
where
\[
h(s,u)=C\left(u^{\frac12-H}(s-u)^{H+\eps-1}
+ (s-u)^{\eps-\frac12} \right).
\]
If
\begin{equation}\label{eq:p-ineq}
p<\frac1{H-\frac12}, \quad p<\frac1{1-H-\eps}, \quad\text{and}\quad p\le\frac1{\frac12-\eps},
\end{equation}
then
\begin{align*}
\int_0^s\abs{h(s,u)}^p\,du
&\le C\int_0^s \left (u^{(\frac12-H)p}(s-u)^{(H+\eps-1)p}
+ (s-u)^{(\eps-\frac12)p} \right)\,du
\\
&= Cs^{(\eps-\frac12)p+1}
\le CT^{(\eps-\frac12)p+1}<\infty.
\end{align*}
Thus, the condition \ref{(B2)} holds with $\delta = H-\eps$ for all $p$ satisfying
\eqref{eq:p-ineq} (in particular, for $p=2$).

According to Lemma~\ref{l:general}, if
$0<\int_\R x^2\,\pi(dx)<\infty$,
then for all $0\le s<t\le T$,
\[
\ex\abs{Y^H_t-Y^H_s}^2 \le C(t-s)^{2(H-\eps)},
\]
and the trajectories of $Y^H$ are $\varkappa$-H\"older continuous for any $\varkappa\in(0,H-\frac12)$.

If $Z$ is a Brownian motion, then, by Lemma~\ref{l:Brownian},
for all $p\ge2$ and all $0\le s<t\le T$,
\[
\ex\abs{Y^H_t-Y^H_s}^p \le C(t-s)^{p(H-\eps)},
\]
and the trajectories of $Y^H$ are $\varkappa$-H\"older continuous for any $\varkappa\in(0,H)$.

\emph{Case $H\in(0,\frac12)$.}
Denote
\[
K_H^{(1)}(t,s) = t^{H-\frac12} s^{\frac12-H}(t-s)^{H-\frac12},
\quad
K_H^{(2)}(t,s) = s^{\frac12-H}\int_s^t u^{H-\frac32}(u-s)^{H-\frac12}\,du.
\]
Then \eqref{eq:MG-kernel} implies that
\[
\abs{K_H(t,s)} \le C\left (K_H^{(1)}(t,s)+K_H^{(2)}(t,s)\right).
\]
According to Remark~\ref{rem:generalization}, we can treat $K_H^{(1)}(t,s)$ and $K_H^{(2)}(t,s)$ separately.
Evidently, the kernel $K_H^{(1)}(t,s)$ satisfies \ref{(B1)} with
$\alpha_1=H-\frac12$, $\beta_1=\frac12-H$, $\gamma_1=H-\frac12$.
Then $\kappa_1=\gamma_1=H-\frac12$, see \eqref{eq:eps}.

In order to bound $K_H^{(2)}(t,s)$, we make a substitution $z=\frac{u-s}s$ in the integral. We get
\begin{align*}
K_H^{(2)}(t,s) &= s^{H-\frac12}\int_0^{\frac{t-s}s} \frac{z^{H-\frac12}}{(1+z)^{\frac32-H}}\,dz
\le s^{H-\frac12}\int_0^{\infty} \frac{z^{H-\frac12}}{(1+z)^{\frac32-H}}\,dz
\\
&= \Beta\left(H+\tfrac12,1-2H\right) s^{H-\frac12}.
\end{align*}
Therefore, $K_H^{(2)}(t,s)$ satisfies \ref{(B1)} with
$\alpha_2=0$, $\beta_2=H-\frac12$, $\gamma_2=0$.
Consequently, $\kappa_2=\alpha_2+\beta_2+\gamma_2=H-\frac12=\kappa_1$.
Thus, $K_H (t,s)$ satisfies \ref{(B1')}, and the corresponding value of $\kappa$ equals $H-\frac12$.

Now let us verify the assumption \ref{(B2)}.
Let $0<u<s<t\le T$. We have
\begin{equation}\label{eq:deltaG}
\abs{K_H(t,u)-K_H(s,u)} \le C\abs{K_H^{(1)}(t,u)-K_H^{(1)}(s,u)} + C\abs{K_H^{(2)}(t,u)-K_H^{(2)}(s,u)}.
\end{equation}
The first term in the right-hand side can be decomposed as follows:
\begin{align}
\MoveEqLeft\abs{K_H^{(1)}(t,u)-K_H^{(1)}(s,u)}
=\abs{t^{H-\frac12} u^{\frac12-H}(t-u)^{H-\frac12}
- s^{H-\frac12} u^{\frac12-H}(s-u)^{H-\frac12}}
\notag\\
&\le u^{\frac12-H}(t-u)^{H-\frac12}\abs{t^{H-\frac12}-s^{H-\frac12}}
+ u^{\frac12-H} s^{H-\frac12} \abs{(t-u)^{H-\frac12}-(s-u)^{H-\frac12}}
\notag\\
&\eqqcolon K_H^{(1,1)}(t,s,u) + K_H^{(1,2)}(t,s,u).
\label{eq:delta-G1}
\end{align}
Let $\eps\in(0,\frac12)$.
For $K_H^{(1,1)}(t,s,u)$ we have
\begin{align}
K_H^{(1,1)}(t,s,u) &= C u^{\frac12-H} (t-u)^{H-\frac12} \int_s^t z^{H-\frac32}\,dz
\notag\\
&\le C u^{\frac12-H}(t-s)^{H-\frac12} s^{H-1+\eps} \int_s^t z^{-\frac12-\eps}\,dz
\notag\\
&= C u^{\frac12-H}(t-s)^{H-\frac12} s^{H-1+\eps} \left (t^{\frac12-\eps} - s^{\frac12-\eps}\right )
\notag\\
&\le C u^{\frac12-H} s^{H-1+\eps} (t-s)^{H-\eps}.
\label{eq:G11}
\end{align}
Similarly,
\begin{align*}
K_H^{(1,2)}(t,s,u) &= C u^{\frac12-H} s^{H-\frac12}  \int_s^t (z-u)^{H-\frac32}\,dz
\\
&\le C u^{\frac12-H} s^{H-\frac12} (s-u)^{\eps-\frac12} \int_s^t (z-s)^{H-1-\eps}\,dz,
\end{align*}
where we have used the inequality
\[
(z-u)^{H-\frac32} = (z-u)^{\eps-\frac12} (z-u)^{H-1-\eps}
\le (s-u)^{\eps-\frac12} (z-s)^{H-1-\eps}.
\]
Therefore,
\begin{equation}\label{eq:G12}
K_H^{(1,2)}(t,s,u)\le C u^{\frac12-H} s^{H-\frac12} (s-u)^{\eps-\frac12} (t-s)^{H-\eps}.
\end{equation}

Let us consider $\abs{K_H^{(2)}(t,u)-K_H^{(2)}(s,u)}$. We have
\[
\abs{K_H^{(2)}(t,u)-K_H^{(2)}(s,u)} = u^{\frac12-H}\int_s^t z^{H-\frac32}(z-u)^{H-\frac12}\,du.
\]
Using the bounds
$(z-u)^{H-\frac12}\le(s-u)^{H-\frac12}$
and
$
z^{H-\frac32}= z^{\eps-\frac12}z^{H-1-\eps}
\le  s^{\eps-\frac12}z^{H-1-\eps},
$
we obtain
\begin{align}
\abs{K_H^{(2)}(t,u)-K_H^{(2)}(s,u)} &\le u^{\frac12-H}s^{\eps-\frac12}(s-u)^{H-\frac12}\int_s^tz^{H-1-\eps}\,du
\notag\\
&= C u^{\frac12-H}s^{\eps-\frac12}(s-u)^{H-\frac12}\left(t^{H-\eps} - s^{H-\eps}\right)
\notag\\
&\le C u^{\frac12-H}s^{\eps-\frac12}(s-u)^{H-\frac12}(t-s)^{H-\eps}.
\label{eq:deltaG2}
\end{align}

Combining \eqref{eq:deltaG}--\eqref{eq:deltaG2}, we get
\[
\abs{K_H(t,u)-K_H(s,u)} \le (t-s)^{H-\eps}h(s,u),
\]
where
\[
h(s,u)=Cu^{\frac12-H}\left ( s^{H-1+\eps}+  s^{H-\frac12} (s-u)^{\eps-\frac12} + s^{\eps-\frac12}(s-u)^{H-\frac12}\right ).
\]
It is straightforward to check that if $p<\frac{1}{\frac12-\eps}$ and $p<\frac{1}{\frac12-H}$, then
\[
\int_0^s\abs{h(s,u)}^p\,du \le C s^{1-(\frac12-\eps)p}\le C T^{(\eps-\frac12)p+1}<\infty.
\]
This means, in particular, that the condition \ref{(B2)} is satisfied with $\delta=H-\eps$ and $p=2$.

According to Lemma~\ref{l:general}, if
$\int_\R x^2\,\pi(dx)<\infty$,
then for all $0\le s<t\le T$,
\[
\ex\abs{Y^H_t-Y^H_s}^2 \le C(t-s)^{2(H-\eps)}.
\]

 If $Z$ is a Brownian motion, then, by Lemma~\ref{l:Brownian},
for all $p\ge2$ and all $0\le s<t\le T$,
\[
\ex\abs{Y^H_t-Y^H_s}^p \le C(t-s)^{p(H-\eps)},
\]
and the trajectories of $Y^H$ are $\varkappa$-H\"older continuous for any $\varkappa\in(0,H)$.
\end{proof}

\begin{remark}
If $H<\frac12$ and $Z$ is a non-Gaussian L\'evy process, then the Kolmogorov--Chentsov theorem does not guarantee continuity of $Y^H$, since $2(H-\eps)<1$.
Moreover, if $Z$ is a L\'evy process without Gaussian component, then
according to \cite[Prop.~3.7]{Tikanmaki}, $Y^H$ has discontinuous sample paths with positive probability.
\end{remark}

\subsubsection{The sub-fractional kernel}

Let us consider another example for a kernel satisfying \ref{(B1)}--\ref{(B2)}, namely
\begin{equation}\label{eq:sfbm-kernel}
L_H(t,s) = C s^{\frac32-H}\left (t^{-1}\left (t^2-s^2\right )^{H-\frac12} +  \int_s^t z^{-2}\left (z^2-s^2\right )^{H-\frac12}\,dz\right),
\end{equation}
where $H\in(0,1)$, $C>0$.
This kernel arises in the compact interval representation of the sub-fractional Brownian motion \cite[Section 2.8]{Mishura_Zili} (see also \cite{Dzhaparidze04}).

Let us consider the Volterra process
\[
U^H_t=\int_0^t L_H(t,s)\,dZ_s,\quad t\in[0,T].
\]
It turns out that its properties are similar to those of the process $Y^H$ in \eqref{eq:Levy-fBm}.

\begin{proposition}
Let $H\in(0,1)$, $\eps\in(0,H)$.
\begin{enumerate}[1.]
\item
Let
$0<\int_\R x^2\,\pi(dx)<\infty$.
Then, for all $0\le s<t\le T$,
\[
\ex\abs{U^H_t-U^H_s}^2 \le C(t-s)^{2(H-\eps)}.
\]
If $H\in(\frac12,1)$, then the trajectories of $U^H$ are $\varkappa$-H\"older continuous for any $\varkappa\in(0,H-\frac12)$.

\item
Let $Z$ be a Brownian motion. Then
for all $p\ge2$ and all $0\le s<t\le T$,
\[
\ex\abs{U^H_t-U^H_s}^p \le C(t-s)^{p(H-\eps)},
\]
and the trajectories of $U^H$ are $\varkappa$-H\"older continuous for any $\varkappa\in(0,H)$.
\end{enumerate}
\end{proposition}

\begin{proof}

\emph{Case $H=\frac12$.}
Observe that $L_H\equiv \const$ in this case, and the statement holds, see the proof of Proposition~\ref{p:fBm}.

\emph{Case $H\in(\frac12,1)$.}
 It is not hard to see that \eqref{eq:sfbm-kernel} can be written in the following form:
\[
L_H(t,s) = C s^{\frac32-H}\int_s^t \left (z^2-s^2\right )^{H-\frac32}\,dz.
\]
Then
\[
L_H(t,s) = C s^{\frac32-H}\int_s^t (z-s)^{H-\frac32}(z+s)^{H-\frac32}\,dz
\le C\int_s^t (z-s)^{H-\frac32}\,dz,
\]
because $(z+s)^{H-\frac32}\le s^{H-\frac32}$.
Therefore,
\[
L_H(t,s)\le C(t-s)^{H-\frac12},
\]
and \ref{(B1)} holds with $\alpha=\beta=0$ and $\gamma=H-\frac12$.

Let us verify \ref{(B2)}.
For $0<u<s<t\le T$ we have
\begin{align*}
\abs{L_H(t,u)-L_H(s,u)} &= C u^{\frac32-H} \int_s^t \left(z^2-u^2\right)^{H-\frac32}\,dz
\\
&= C u^{\frac32-H} \int_s^t (z-u)^{H-\frac32} (z+u)^{H-\frac32} \,dz.
\end{align*}
Using the bound
\[
(z+u)^{H-\frac32} = (z+u)^{-1} (z+u)^{H-\frac12}
\le u^{-1} (2z)^{H-\frac12},
\]
we get
\[
\abs{L_H(t,u)-L_H(s,u)} \le C u^{\frac12-H}\! \int_s^t (z-u)^{H-\frac32} z^{H-\frac12} \,dz
=C\abs{K_H(t,u)-K_H(s,u)},
\]
see \eqref{eq:example-bound-dif}.
Thus, the condition \ref{(B2)} holds with $\delta = H-\eps$ for all $p$ satisfying
\eqref{eq:p-ineq} (in particular, for $p=2$),
see the proof of Proposition~\ref{p:fBm}.
Similarly to the case of the Molchan--Golosov kernel, we can conclude that the proposition holds in the case $H>\frac12$.

\emph{Case $H\in(0,\frac12)$.}
It follows from \eqref{eq:sfbm-kernel} that
\[
\abs{L_H(t,s)} \le C\left (L_H^{(1)}(t,s)+L_H^{(2)}(t,s)\right).
\]
where
\[
L_H^{(1)}(t,s) = s^{\frac32-H}t^{-1}\left (t^2-s^2\right )^{H-\frac12},
\quad
L_H^{(2)}(t,s) = s^{\frac32-H}  \int_s^t z^{-2}\left (z^2-s^2\right )^{H-\frac12}\,dz.
\]
Applying the estimate
\begin{align}
\left (t^2-s^2\right )^{H-\frac12}
&=(t-s)^{H-\frac12}(t+s)^{H-\frac12}= (t-s)^{H-\frac12}(t+s)^{H+\frac12}(t+s)^{-1}
\notag\\
&\le (t-s)^{H-\frac12}(2t)^{H+\frac12}s^{-1}
= C t^{H+\frac12}s^{-1} (t-s)^{H-\frac12},
\label{eq:aux-bound}
\end{align}
we obtain
\begin{equation}\label{eq:F1}
L_H^{(1)}(t,s)
\le C s^{\frac12-H}t^{H-\frac12}(t-s)^{H-\frac12}
= C K_H^{(1)}(t,s).
\end{equation}
For the term $L_H^{(2)}(t,s)$ we also use \eqref{eq:aux-bound} and arrive at
\begin{equation}\label{eq:F2}
L_H^{(2)}(t,s) \le C s^{\frac12-H}  \int_s^t z^{H-\frac32} (z-s)^{H-\frac12}\,dz
=C  K_H^{(2)}(t,s).
\end{equation}
From the bounds \eqref{eq:F1} and \eqref{eq:F2} we deduce that the condition \ref{(B1')} is satisfied with the same constants $\alpha_i$, $\beta_i$, $\gamma_i$, $i=1,2$, as in the case of the kernel $K_H(t,s)$, see the proof of Proposition~\ref{p:fBm}.

Now we consider the difference
\[
\abs{L_H(t,u)-L_H(s,u)} \le C\abs{L_H^{(1)}(t,u)-L_H^{(1)}(s,u)} + C\abs{L_H^{(2)}(t,u)-L_H^{(2)}(s,u)},
\]
where $0<u<s<t\le T$.
For the first term in the right-hand side we have
\begin{align*}
\MoveEqLeft[1]
\abs{L_H^{(1)}(t,u)-L_H^{(1)}(s,u)}
= u^{\frac32-H}\abs{t^{-1}\left (t^2-u^2\right )^{H-\frac12}
- s^{-1}\left (s^2-u^2\right )^{H-\frac12}}
\\
&\le u^{\frac32-H}\left (t^2-u^2\right )^{H-\frac12} \abs{t^{-1}-s^{-1}}
+u^{\frac32-H}s^{-1}\abs{\left (t^2-u^2\right )^{H-\frac12}-\left (s^2-u^2\right )^{H-\frac12}}
\\
&\eqqcolon L_H^{(1,1)}(t,s,u) + L_H^{(1,2)}(t,s,u).
\end{align*}
Consider
\[
L_H^{(1,1)}(t,s,u) = C u^{\frac32-H}(t-u)^{H-\frac12}(t+u)^{H-\frac12} \int_s^t z^{-2}\,dz.
\]
Since
\[
(t+u)^{H-\frac12} \le u^{H-\frac12},
\quad
z^{-2} = z^{H-\frac32}z^{-H-\frac12} \le z^{H-\frac32}u^{-H-\frac12}
\]
we see that
\[
L_H^{(1,1)}(t,s,u) \le C u^{\frac12-H}(t-u)^{H-\frac12} \int_s^t z^{H-\frac32}\,dz
\le C K_H^{(1,1)}(t,s,u).
\]
by \eqref{eq:G11}.
The term $L_H^{(1,2)}(t,s,u)$ can be rewritten as follows:
\begin{align*}
L_H^{(1,2)}(t,s,u) &= C u^{\frac32-H}s^{-1}\int_{s^2}^{t^2}\left (z-u^2\right )^{H-\frac32}\,dz
\\
&= C u^{\frac32-H}s^{-1}\int_{s}^{t}\left (x^2-u^2\right )^{H-\frac32} x\,dx
\\
&= C u^{\frac32-H}s^{-1}\int_{s}^{t}(x-u)^{H-\frac12}(x+u)^{H-\frac32} x\,dx
\end{align*}
Using the bound,
\begin{align*}
(x+u)^{H-\frac32} x
&= (x+u)^{H-\frac32} x s^{H-\frac12} s^{\frac12-H}
\\
&\le (x+u)^{H-\frac32} (x+u) s^{H-\frac12} (x+u)^{\frac12-H}
= s^{H-\frac12}
\end{align*}
we obtain
\begin{align*}
L_H^{(1,2)}(t,s,u) &\le C u^{\frac32-H}s^{H-\frac32}\int_{s}^{t}(x-u)^{H-\frac12}\,dx
\\
&= C \frac{u}{s} \,K_H^{(1,2)}(t,s,u)
\le C K_H^{(1,2)}(t,s,u).
\end{align*}

Finally, applying the inequality \eqref{eq:aux-bound}, we get
\begin{multline*}
\abs{L_H^{(2)}(t,u)-L_H^{(2)}(s,u)} = u^{\frac32-H}  \int_s^t z^{-2}\left (z^2-u^2\right )^{H-\frac12}\,dz
\\
\le C u^{\frac12-H}  \int_s^t z^{H-\frac32} (z-u)^{H-\frac12}\,dz
= C \abs{K_H^{(2)}(t,u)-K_H^{(2)}(s,u)}.
\end{multline*}

Thus, we have established that
\begin{multline*}
\abs{L_H(t,u)-L_H(s,u)}
\le  C K_H^{(1,1)}(t,s,u) + C K_H^{(1,2)}(t,s,u)
\\
+ C \abs{K_H^{(2)}(t,u)-K_H^{(2)}(s,u)}.
\end{multline*}
The proof is concluded by applying the bounds and the arguments from the proof of  Proposition~\ref{p:fBm}.
\end{proof}

\subsection{Sonine pairs and  two kinds of Volterra--Gaussian processes}
\label{sec:3.5}

Hereafter we discuss some family of kernels providing in turn Volterra--Gaussian processes with good paths regularity.
The characterization of the kernels is based on the so-called Sonine pairs.
As a motivation, consider the compact interval representation \eqref{fBm} of the fractional Brownian motion, where the kernel is given by \eqref{eq:MG-kernel}.
We shall consider \eqref{eq:MG-kernel} in the two cases $H\in(\frac12,1)$ and $H\in(0,\frac12)$.
This will lead to different kind of considerations on the family of kernels.

\medskip

\textbf{(a)}
Let us consider first $H\in(\frac12,1)$.
In this case, the kernel $K_H$ can be simplified to
\begin{equation}\label{eq:A1}
K_H(t,s) = \left(H-\tfrac12\right) C_H s^{\frac12-H} \int_s^t u^{H-\frac12} (u-s)^{H-\frac32}\,du.
 \end{equation}
This leads us to consider the following Gaussian process
\begin{equation}\label{eq:gfbm}
Y_t = \int_0^t K(t,s) \,dW_s, \quad t\in[0,T],
\end{equation}
where $W=\set{W_t,t\in[0,T]}$ is a Wiener process, and the Volterra kernel $K(t,s)$ has the following form
\begin{equation}\label{eq:kernK}
K(t,s) = a(s) \int_s^t b(u) c(u-s)\,du.
\end{equation}
The functions $a,b,c\colon[0,T]\to\R$ are measurable and satisfy the following assumptions
\begin{enumerate}[label=\textbf{(C\arabic*)}]
\item\label{(C1)} Functions
$a\in L_p([0,T])$, $b\in L_q([0,T])$, and $c\in L_r([0,T])$
for $p\in[2,\infty]$, $q\in[1,\infty]$, $r\in[1,\infty]$
such that $1/p+1/q+1/r\le 3/2$.
\item\label{(C2)}
Functions $a$, $b$ are positive a.\,e.\ on $[0,T]$.
\item\label{(C3)}
Function $c$ creates a Sonine pair with some $h\in L_1([0,T])$.
\end{enumerate}
Recall the definition of Sonine pairs as given in \cite{MSS}.
\begin{definition}
The function $c$ creates a Sonine pair on the interval $[0,T]$ with some function $h\in L_1([0,T])$ if, for any $t\in[0,T]$,
\[
\int_0^t c(t-s) h(s)\, ds = 1.
\]
\end{definition}
It was established in \cite{MSS} that under the assumption \ref{(C1)},
\[
\sup_{t\in[0,T]} \norm{K(t,\cdot)}_{L_2([0,t])}<\infty.
\]
 This means that for any Wiener process $W=\set{W_t,t\in[0,T]}$, the process
\[
Y_t = \int_0^t K(t,s) \,dW_s, \quad t\in[0,T],
\]
is well defined, see \cite[Thm.~1]{MSS}.

\begin{remark}
If $H\in(\frac12,1)$, $a(s) = C s^{\frac12-H}$,
$b(s)=s^{H-\frac12}$,
$c(s) = s^{H-\frac32}$,
then $K(t,s)$ is the Molchan--Golosov kernel \eqref{eq:MG>1/2}, hence $Y$ is a fractional Brownian motion with the Hurst index $H$.
 Moreover, in this case the assumptions \ref{(C1)}--\ref{(C3)} are satisfied, see \cite{MSS}.  Therefore, the kernel $K$ is an analog of the kernel $K_H$ with $H>\frac12$. In this case $h(s)=s^{\frac12-H}.$
Other examples of Sonine pairs $(c,h)$ are given in \cite{MSS}.
\end{remark}

Let us consider the operator $\K$ associated with the kernel $K(t,s)$ in \eqref{eq:kernK}:
\begin{equation}\label{eq:E}
\K f(t) = \int_0^t K(t,s) f(s)\,ds
=\int_0^t a(s)\int_s^t b(u) c(u-s)\,du\, f(s)\,ds.
\end{equation}

In order to find an inverse operator to $\K$, let us apply the elements of ``fractional'' calculus related to Sonine pair $(c, h)$. More precisely, we use the notions similar to notions of the fractional integral and the fractional derivative, as given in Definition \ref{def:fracintfracder} from Appendix, see also \cite{MSS}.

In terms of the  fractional integral $I_{0+}^c$ from Definition \ref{def:fracintfracder}, the operator $\K$ can be rewritten   as follows:
\begin{equation}\label{eq:int}
\K f(t) = \int_0^t b(u)\int_0^s a(s)  c(u-s) f(s)\,ds\,du
=\int_0^t b(u) I_{0+}^c (af)(u)\,du.
\end{equation}

\begin{lemma}\label{lem6}
Consider the equation
\[
\K f(t) = \int_0^t a(s) \int_s^t b(u) c(u-s)\,du f(s)\,ds
=\int_0^t u(z)\,dz, \quad t\in[0,T].
\]
Then its solution has a form
\begin{equation}\label{eq:sol}
f(t) = a^{-1}(t) D_{0+}^h(u b^{-1})(t),
\end{equation}
under the assumption that the right-hand side of \eqref{eq:sol} is well-defined and $D_{0+}^h(u b^{-1})\in L_1([0,T])$.
Here $D_{0+}^h$ stands for fractional derivative, see Definition~\ref{def:fracintfracder}.
\end{lemma}

\begin{proof}
According to \eqref{eq:int},
\[
b(t) I_{0+}^c(af)(t) = u(t) \quad\text{a.\,e.}
\]
or
\begin{equation}\label{eq:**}
I_{0+}^c(af)(t) = b^{-1}(t)u(t).
\end{equation}
Assume that $af\in L_1([0,T])$ and apply Lemma~\ref{lem8}, item \ref{l1-i} to \eqref{eq:**}.
As a result, we arrive to
\[
af(t) = D_{0+}^h (b^{-1}u)(t),
\]
and the proof follows.
\end{proof}

As already mentioned, condition \ref{(C1)} is sufficient for the existence of process $Y$.
However, in order to guarantee its H\"older continuity, a stronger assumption is required.
The following proposition summarizes the results in Lemma~1 and Theorem~3 of
\cite{MSS}.

\begin{proposition}\label{prop:holder}
\begin{enumerate}[1.]
\item
Let the coefficients $a$, $b$, $c$ satisfy the assumption
\begin{enumerate}[label=\textbf{(C\arabic*)},start=4]
\item\label{(C4)}
$a\in L^p([0,T])$, $b\in L^q([0,T])$, $c\in L^r([0,T])$, where $p\ge 2$, $q,r\ge 1$,
$\frac1p+ \frac1r\le\frac12$, and $\frac1p+\frac1q+\frac1r\le 1+\varepsilon$ for some $\varepsilon\in (0,1/2)$.
\end{enumerate}
Then the stochastic process $Y$ has a modification satisfying H\"{o}lder condition up to order $\nu=\frac32-\frac1p-\frac1q-\frac1r>1/2-\varepsilon$.

\item
Let the coefficients $a$, $b$, $c$ satisfy the assumption
\begin{enumerate}[resume*]
\item\label{(C5)}
for any $t_1\ge0$, $t_2\ge0$, $t_1+t_2<T$,
\begin{gather*}
a\in L^p([0,T]) \cap L^{p_1}([t_1,T]), \text{ where } 2\le p\le p_1,
\\
b\in L^q([0,T]) \cap L^{q_1}([t_1+t_2,T]), \text{ where } 1< q\le q_1,
\\
c\in L^r([0,T]) \cap L^{r_1}([t_2,T]), \text{ where } 1\le r\le r_1,
\end{gather*}
$\frac1p+\frac1q+\frac1r\le\frac32$, and
$\frac{1}{q_1}+\max\left(\frac12,\frac1p+\frac{1}{r_1},\frac{1}{p_1}+\frac{1}{r}\right)<1$.
\end{enumerate}
Then the process $Y$ on any interval $[t_1+t_2,T]$ has a modification that satisfies
H\"older condition up to order
$\mu = \frac32 - \frac{1}{q_1} - \max\left(\frac12,\frac1p+\frac{1}{r_1},\frac{1}{p_1}+\frac{1}{r}\right)>1/2$.
\end{enumerate}
\end{proposition}

In \cite {MSS} details about the value of these conditions for fractional Brownian motion are given.
Briefly, condition \ref{(C4)} supplies its H\"{o}lder property up to order $1/2$, and condition \ref{(C5)} supplies H\"{o}lder property up to order $H$ on any interval separated from zero.

\medskip

\textbf{(b)}
In the present paper, we consider the kernel \eqref{eq:MG-kernel} with Hurst index $H\in(0, 1/2)$.
Then we introduce its generalization in the form
\begin{equation}\label{eq:genker}
\widehat K(t,s) = \hat a(s)\left[\hat b(t) \hat c (t-s)
- \int_s^t \hat b'(u) \hat c(u-s)\,du \right],
\end{equation}
where   $\hat a, \hat b, \hat c  \colon[0,T]\to\R$ are measurable functions. In what follows, we assume that the following conditions hold.
\begin{enumerate}[label={\bf ($\mathbf{\widehat C \arabic*}$)}]
\item\label{(hatC1)}
The function $\hat a$ is nondecreasing, $\hat b$ is absolutely continuous, $\hat a \hat b$ is bounded, $\hat c \in L_2([0,T])$, and
\[
A(T)\coloneqq \int_0^T\!\!\int_0^T \abs{\hat b'(u)} \abs{\hat b'(z)}
\int_0^{u\wedge z}\! \hat a^2 (s) \abs{\hat c(u-s)} \abs{\hat c(z-s)} \,ds\,du\,dz<\infty.
\]
\item\label{(hatC2)}
Functions $\hat{a}$, $\hat{b}$ are positive a.\,e.\ on $[0,T]$.
\item\label{(hatC3)}
Function $\hat{c}$ creates a Sonine pair with some $\hat{h}\in L_1([0,T])$.
\end{enumerate}

\begin{remark}
Sufficient condition for \ref{(hatC1)} is
\begin{enumerate}[label={\bf ($\mathbf{\widehat C \arabic*'}$)}]
\item\label{(hatC1')}
$\displaystyle \int_0^T \abs{\hat b'(u)} \left(
\int_0^{u} \hat a^2 (z) \, \hat c^2(u-z) \,dz\right)^{\frac12}du<\infty$.
\end{enumerate}
Indeed, under \ref{(hatC1')}
\begin{align*}
\MoveEqLeft
\int_0^T\!\!\int_0^T \abs{\hat b'(u)} \abs{\hat b'(z)}
\int_0^{u\wedge z} \hat a^2 (s) \abs{\hat c(u-s)} \abs{\hat c(z-s)} \,ds\,du\,dz
\\
&\le \int_0^T\!\!\int_0^T \abs{\hat b'(u)} \abs{\hat b'(z)}
\left(\int_0^{u\wedge z} \hat a^2 (s) \abs{\hat c(u-s)}^2\,ds\right)^{\frac12}
\\
&\quad\times
\left(\int_0^{u\wedge z} \hat a^2 (s) \abs{\hat c(z-s)}^2\,ds\right)^{\frac12}
du\,dz
\\
&\le \left(\int_0^T \abs{\hat b'(u)}\left(\int_0^u \hat a^2 (s)
\abs{\hat c(u-s)}^2\,ds\right)^{\frac12} du \right)^2<\infty.
\end{align*}
\end{remark}

\begin{remark}
We observe that in the case of fractional Brownian motion with $H<1/2$ it holds that $\hat a(s) = Cs^{\frac12-H}$,
$\hat b(s) = \hat c(s) =  s^{H-\frac12}$ and $\hat h(s) = s^{-\frac12-H}$, so, these functions indeed satisfy conditions  \ref{(hatC1)}--\ref{(hatC3)}.
Indeed, from the remark above we can see that
\begin{align*}
\MoveEqLeft
\int_0^T \abs{\hat b'(u)} \left(\int_0^{u} \hat a^2 (s) \, \abs{\hat c(u-s)}^2 \,ds\right)^{\frac12}du
\\
&= C\left(\tfrac12-H\right) \int_0^T u^{H-\frac32}
\left(\int_0^{u} s^{1-2H} (u-s)^{2H-1} \,ds\right)^{\frac12}du
\\
&= C\left(\tfrac12-H\right) \int_0^T u^{H-1}\,du < \infty.
\end{align*}
\end{remark}

\begin{lemma}
Under assumption \ref{(hatC1)} we have that
\[
\sup_{t\in[0,T]} \norm{\widehat K(t,\cdot)}_{L_2([0,t])} < \infty,
\]
and for any Wiener process $W = \set{W_t, t\in[0,T]}$ a process
\[
\widehat{Y}_t = \int_0^t \widehat{K}(t,s)\,dW_s,
\quad t\in[0,T],
\]
is well defined.
\end{lemma}

\begin{proof}
Obviously,
\begin{align*}
\norm{\widehat K (t, \cdot)}_{L_2([0,t])}^2
&\le C\hat b^2(t) \int_0^t \hat a^2(s) \hat c^2(t-s)\,ds
\\
&\quad+ C \int_0^t \hat a^2(s) \left(\int_s^t\hat b(u) \hat c(u-s)\,du\right)^2 ds.
\end{align*}
If $\hat a \hat b$ is bounded, $\hat a$ is nondecreasing, and
$\hat c \in L_2([0,T])$, then
\begin{align*}
\hat b^2(t) \int_0^t \hat a^2(s) \hat c^2(t-s)\,ds
&\le \hat b^2(t) \hat a^2(t) \int_0^t \hat c^2(t-s)\,ds
\\
&\le \bigl(\hat a\hat b\bigr)^2(t) \norm{\hat c}^2_{L_2([0,T])}
<\infty.
\end{align*}
Furthermore,
\begin{align*}
\MoveEqLeft
\int_0^t \hat a^2(s) \left(\int_s^t\hat b(u) \hat c(u-s)\,du\right)^2 ds
\\
&\le \int_0^t \hat a^2(s) \int_s^t \hat b(u) \abs{\hat c(u-s)}\,du
\int_s^t \hat b(v) \abs{\hat c(v-s)}\,dv\,ds
\\
&= \int_0^t\!\!\int_0^t \hat b(u)\hat b(v) \int_0^{u\wedge v}
\hat a^2(s) \abs{\hat c(u-s)} \abs{\hat c(v-s)}\,ds\,du\,dv
\le A(T)<\infty,
\end{align*}
and the proof follows.
\end{proof}

Let us now consider the operator $\widehat\K$ associated with the kernel
$\widehat K(t,s)$ in \eqref{eq:genker} (similarly to the operator $\K$ from \eqref{eq:E} associated with $K(t,s)$).
In this case $\widehat\K$ has the form
\begin{align*}
\widehat\K f(t) &= \int_0^t \widehat K(t,s) f(s) \,ds
\\
&= \hat b(t)I_{0+}^c (\hat a f)(t)
- \int_0^t (\hat a f)(s) \int_s^t \hat b'(u) \hat c(u-s)\,du\,ds,
\quad f\in L_2([0,T]),
\end{align*}
and under the assumptions   \ref{(hatC1)}--\ref{(hatC3)}
we can apply the Fubini theorem and get
\begin{align*}
\widehat\K f(t) &= \hat b(t)I_{0+}^c (\hat a f)(t)
- \int_0^t \hat b'(u) \int_0^u \hat c(u-s)(\hat a f)(s)\,ds\,du
\\
&= \int_0^t \hat b(u)\, \frac{d}{du} \int_0^u \hat c(u-z)(\hat a f)(z)\,dz\,du
 = \int_0^t \hat b(u) D_{0+}^{\hat c}(\hat a f)(u)\,du.
\end{align*}
Consider the following Gaussian process
\begin{equation}\label{eq:gfbm1}
\widehat{Y}_t = \int_0^t \widehat{K}(t,s) \,dW_s, \quad t\in[0,T],
\end{equation}
where $W=\set{W_t,t\in[0,T]}$ is a Wiener process. Under assumptions \ref{(hatC1)}--\ref{(hatC3)} it is well defined on $[0,T]$.
Taking Lemma~\ref{lem8} from Appendix into account, it is easy to establish, similarly to Lemma~\ref{lem6}, the following result.

\begin{lemma}
Consider the equation
\[
\widehat{\mathcal K}f (t) = \int_0^t u(z)\,dz, \quad z\in[0,T].
\]
Then its solution has a form
\[
f(t) = \hat a^{-1} I_{0+}^h\left(\hat b^{-1} u\right)(t).
\]
\end{lemma}

Furthermore, we prove the following result on the H\"older continuity of paths.

\begin{theorem}
Let the conditions \ref{(hatC1)}--\ref{(hatC3)} hold, together with the following assumptions:
\begin{enumerate}[label={\bf ($\mathbf{\widehat C \arabic*}$)},start=4]
\item\label{(hatC4)}
$\displaystyle \abs{\hat a(t)\hat b'(t)} \le C t^{-1}$, $t\in[0,T]$;

\item\label{(hatC5)}
there exists $\gamma\in(0,2)$ such that
\[
\int_0^t \hat c^2(s)\,ds \le Ct^\gamma, \quad t\in[0,T],
\]
and
 \[
\int_0^{T-t} \bigl(\hat c(t+s)-\hat c(s)\bigr)^2\,ds \le Ct^\gamma, \quad t\in[0,T].
\]
\end{enumerate}
Then the trajectories of the process $\widehat Y$ satisfy $\delta$-H\"older condition a.\,s. for any $\delta\in(0,\gamma/2)$.
\end{theorem}

\begin{remark}
In the case when
$\hat a(s) = C s^{\frac12-H}$,
$\hat b(s) = \hat c(s) = s^{H-\frac12}$
we have that
$\hat a(s) \hat b'(s) = C s^{-1}$,
$\int_0^t\hat c^2(s)\,ds = C t^{2H}$,
\begin{align*}
\MoveEqLeft
\int_0^{T-t}\left(\hat c(t+s) - \hat c(s)\right)^2 ds
= \int_0^{T-t}\left((t+s)^{H-\frac12} - s^{H-\frac12}\right)^2 ds
\\
&= t^{2H}\int_0^{\frac{T}{t}-1}\left((1+z)^{H-\frac12} - z^{H-\frac12}\right)^2 dz
< t^{2H}\int_0^\infty \left((1+z)^{H-\frac12} - z^{H-\frac12}\right)^2 dz
\\
&\le Ct^{2H},
\end{align*}
because
$(1+z)^{H-\frac12} - z^{H-\frac12} \sim z^{H-\frac32}$, $z\to\infty$,
and so
$\int_0^\infty \left((1+z)^{H-\frac12} - z^{H-\frac12}\right)^2 dz\le C$.
Therefore, in this case we can put $\gamma = 2H$, and \ref{(hatC4)}--\ref{(hatC5)} hold.
\end{remark}

\begin{proof}
For $t_1<t_2$,
\begin{align*}
\MoveEqLeft
\ex\left(\widehat Y_{t_2}-\widehat Y_{t_1}\right)^2
=\ex\left(\int_0^{t_1}\left(\widehat K(t_2,s)-\widehat K(t_1,s)\right)\,dW_s + \int_{t_1}^{t_2}\widehat K(t_2,s)\,dW_s\right)^2
\\
&=\int_0^{t_1}\left(\widehat K(t_2,s)-\widehat K(t_1,s)\right)^2\,ds + \int_{t_1}^{t_2}\widehat K^2(t_2,s)\,ds
\\
&\le 2\Biggl(\int_0^{t_1}\hat a^2(s)\left(\hat b(t_2) \hat c (t_2-s) - \hat b(t_1) \hat c (t_1-s)\right)^2\,ds
\\
&\quad + \hat b^2(t_2) \int_{t_1}^{t_2}\hat a^2(s) \hat c^2 (t_2-s)\,ds
+\int_0^{t_1}\hat a^2(s)\left(\int_{t_1}^{t_2} \hat b'(u) \hat c(u-s)\,du\right)^2\,ds
\\
&\quad+\int_{t_1}^{t_2}\hat a^2(s)\left(\int_{s}^{t_2} \hat b'(u) \hat c(u-s)\,du\right)^2\,ds
\Biggr)
\eqqcolon 2(I_1+I_2+I_3+I_4)
\end{align*}
Let us show that each term in the right-hand side is bounded by $C(t_2-t_1)^\gamma$.
We make the analysis term by term.

1. The first term can be rewritten as follows:
\begin{align}
I_1&=
\int_0^{t_1}\hat a^2(s)\left(\hat b(t_2) \hat c (t_2-s) - \hat b(t_1) \hat c (t_1-s)\right)^2\,ds
\notag
\\
&\le 2\hat b^2(t_2) \int_0^{t_1}\hat a^2(s)\left(\hat c (t_2-s) - \hat c (t_1-s)\right)^2\,ds
\notag
\\
&\quad+ 2\left(\hat b(t_2) - \hat b(t_1)\right)^2
\int_0^{t_1}\hat a^2(s)\hat c^2(t_1-s)\,ds=:J_1+J_2.
\label{eq:1term}
\end{align}
The first term in the right-hand side of \eqref{eq:1term} is bounded as follows:
\begin{align*}
\MoveEqLeft
\hat b^2(t_2) \int_0^{t_1}\hat a^2(s)\left(\hat c (t_2-s) - \hat c (t_1-s)\right)^2\,ds
\\
&\le \hat b^2(t_2)\hat a^2(t_2)\int_0^{t_1}\left(\hat c (t_2-s) - \hat c (t_1-s)\right)^2\,ds
\\
&\le C \int_0^{t_1}\left(\hat c(t_2-t_1+z) - \hat c(z)\right)^2\,dz
\le C(t_2-t_1)^\gamma,
\end{align*}
and  the second one can be bounded as follows:
\begin{align*}
\MoveEqLeft
\left(\hat b(t_2) - \hat b(t_1)\right)^2
\int_0^{t_1}\hat a^2(s)\hat c^2(t_1-s)\,ds
=\int_0^{t_1}\hat a^2(s)\hat c^2(t_1-s)
\left(\int_{t_1}^{t_2}\hat b'(v)\,dv\right)^2ds
\\
&\le \int_0^{t_1}\hat c^2(t_1-s)
\left(\int_{t_1}^{t_2}\abs{\hat a(v)\hat b'(v)}\,dv\right)^2ds
\le C t_1^\gamma
\left(\int_{t_1}^{t_2}v^{-1}\,dv\right)^2
\\
&= C\left(\int_{t_1}^{t_2}t_1^{\frac\gamma2}v^{-1}\,dv\right)^2
\le C\left(\int_{t_1}^{t_2}v^{\frac\gamma2-1}\,dv\right)^2
\le C\left(t_2^{\frac\gamma2}-t_1^{\frac\gamma2}\right)^2
\le C\left(t_2-t_1\right)^\gamma.
\end{align*}
Here we have used the monotonicity of $ {\hat a}$ and then the conditions \ref{(hatC4)} and \ref{(hatC5)}.

2. The second term can be bounded with the help of conditions \ref{(hatC1)} and \ref{(hatC5)}:
\begin{align*}
I_2&=
\hat b^2(t_2) \int_{t_1}^{t_2}\hat a^2(s) \hat c^2 (t_2-s)\,ds
\le \hat a^2(t_2) \hat b^2(t_2)  \int_{t_1}^{t_2} \hat c^2 (t_2-s)\,ds
\\
&\le C \int_{0}^{t_2-t_1} \hat c^2 (z)\,dz
\le C\left(t_2-t_1\right)^\gamma.
\end{align*}

3. By Fubini's theorem and monotonicity of $ {\hat a}$, the third term can be estimated as follows:
\begin{align*}
I_3&=
\int_0^{t_1}\hat a^2(s)\left(\int_{t_1}^{t_2} \hat b'(u) \hat c(u-s)\,du\right)^2\,ds
\\
&= \int_0^{t_1}\hat a^2(s) \int_{t_1}^{t_2} \hat b'(u) \hat c(u-s)\,du\int_{t_1}^{t_2} \hat b'(v) \hat c(v-s)\,dv\,ds
\\
&= \int_{t_1}^{t_2}\!\!\int_{t_1}^{t_2} \hat b'(u) \hat b'(v)
\int_0^{t_1}\hat a^2(s)\hat c(u-s) \hat c(v-s)\,ds\,du\,dv
\\
&\le \int_{t_1}^{t_2}\!\!\int_{t_1}^{t_2} \abs{\hat a(u)\hat b'(u)} \abs{\hat a(v)\hat b'(v)}
\int_0^{t_1}\abs{\hat c(u-s) \hat c(v-s)}\,ds\,du\,dv.
\end{align*}
Then applying successively the condition \ref{(hatC4)}, the Cauchy--Schwarz inequality and the condition \ref{(hatC5)} we obtain:
\begin{align*}
I_3&\le C\int_{t_1}^{t_2}\!\!\int_{t_1}^{t_2}u^{-1}v^{-1}
\int_0^{t_1}\abs{\hat c(u-s) \hat c(v-s)}\,ds\,du\,dv
\\
&\le C\int_{t_1}^{t_2}\!\!\int_{t_1}^{t_2}u^{-1}v^{-1}
\left(\int_0^{t_1}\hat c^2(u-s)\,ds\right)^{\frac12}
\left(\int_0^{t_1}\hat c^2(v-s)\,ds\right)^{\frac12}du\,dv
\\
&\le C\int_{t_1}^{t_2}\!\!\int_{t_1}^{t_2}u^{\frac\gamma2-1}v^{\frac\gamma2-1}
\,du\,dv
\le C\left(t_2^{\frac\gamma2}-t_1^{\frac\gamma2}\right)^2
\le C\left(t_2-t_1\right)^\gamma.
\end{align*}

4. The fourth term can be bounded similarly to the third one:
\begin{align*}
I_4&=
\int_{t_1}^{t_2}\hat a^2(s)\left(\int_{s}^{t_2} \hat b'(u) \hat c(u-s)\,du\right)^2\,ds
\\
&=\int_{t_1}^{t_2}\!\!\int_{s}^{t_2}\!\!\int_{s}^{t_2}\hat a^2(s)
\hat b'(u) \hat c(u-s)\hat b'(v) \hat c(v-s)\,du\,dv\,ds
\\
&\le\int_{t_1}^{t_2}\!\!\int_{s}^{t_2}\!\!\int_{s}^{t_2}
\abs{\hat a(u)\hat b'(u)}\abs{\hat a(v)\hat b'(v)}
\abs{\hat c(u-s)\hat c(v-s)} du\,dv\,ds
\\
&\le C\int_{t_1}^{t_2}\!\!\int_{s}^{t_2}\!\!\int_{s}^{t_2}
u^{-1}v^{-1}\abs{\hat c(u-s)\hat c(v-s)} du\,dv\,ds
\\
&= C\int_{t_1}^{t_2}\!\!\int_{t_1}^{t_2} u^{-1}v^{-1}
\int_{t_1}^{u\wedge v}\abs{\hat c(u-s)\hat c(v-s)} ds\,du\,dv
\\
&\le
C\int_{t_1}^{t_2}\!\!\int_{t_1}^{t_2} u^{-1}v^{-1}
\left(\int_{t_1}^{u\wedge v}\hat c^2(u-s)\,ds\right)^{\frac12}
\left(\int_{t_1}^{u\wedge v}\hat c^2(v-s)\,ds\right)^{\frac12}\,du\,dv
\\
&\le
C\int_{t_1}^{t_2}\!\!\int_{t_1}^{t_2} u^{-1}v^{-1}
\left(\int_{t_1}^{u}\hat c^2(u-s)\,ds\right)^{\frac12}
\left(\int_{t_1}^{v}\hat c^2(v-s)\,ds\right)^{\frac12}\,du\,dv
\\
&\le
C\int_{t_1}^{t_2}\!\!\int_{t_1}^{t_2} u^{-1}v^{-1}
\left(u-t_1\right)^{\frac\gamma2}
\left(v-t_1\right)^{\frac\gamma2}\,du\,dv
\\
&\le
C\int_{t_1}^{t_2}\!\!\int_{t_1}^{t_2}
\left(u-t_1\right)^{\frac\gamma2-1}
\left(v-t_1\right)^{\frac\gamma2-1}\,du\,dv
\le
C\left(t_2-t_1\right)^\gamma.
\end{align*}
Combining the bounds we get
\[
\ex\left(\widehat Y_{t_2}-\widehat Y_{t_1}\right)^2
\le C\left(t_2-t_1\right)^\gamma,
\]
whence the result follows.
\end{proof}

\section{Equations with locally Lipschitz drift of linear growth}
\label{sec:4}

In this section we study stochastic differential equations with additive Volterra--L\'evy noise.
The noise considered has H\"older regularity of the paths as discussed in the first part of this work.
We shall adopt pathwise considerations and, for this reason, we start the study taking deterministic equations into account, then we move to discuss the stochastic cases.

\subsection{Deterministic equation}

To begin, consider nonrandom functions. Namely, let $T>0$ be fixed, $f=f(t)$, $t\in[0,T]$, and coefficient $u=u(x)$, $x\in\R$, be the measurable functions. Introduce the equation
\begin{equation}\label{eq:deterministic}
X_t=\int_0^t u(X_s)\,ds + f(t),\; t\in[0,T],\; X|_{t=0}=X_0\in \R.
\end{equation}

\begin{lemma}\label{l:ex-un_determ}
Let any of two following groups of conditions hold.
\begin{enumerate}[label=\textbf{(D\arabic*)}]
\item\label{(D1)}
\begin{enumerate}[1)]
\item
The coefficient $u$ is Lipschitz: there exists $C>0$ such that for any $x,y\in\R$,
\[
\abs{u(x)-u(y)} \le C \abs{x-y}.
\]
\item
The function $f$ is locally integrable.
\end{enumerate}

\item\label{(D2)}
 \begin{enumerate}[1)]
\item
The coefficient $u$ is of linear growth: there exists $C>0$ such that for any $x\in\R$,
\[
\abs{u(x)}\le C(1+\abs{x}).
\]
\item
the coefficient $u$ is locally Lipschitz: for any $R>0$ there exists $C_R>0$ such that for any $x,y\in\R$, $\abs{x},\abs{y}<R$,
\[
\abs{u(x)-u(y)} \le C_R \abs{x-y}.
\]
\item The function $f$ is locally bounded.
\end{enumerate}
\end{enumerate}
Then the equation \eqref{eq:deterministic} has a unique solution $X$ on $[0,T]$.
If condition \ref{(D1)} holds, then $X$ is locally integrable.
If condition \ref{(D2)} holds, then $X$ is locally bounded.
\end{lemma}

\begin{proof}
First, we assume that \ref{(D1)} holds.
Let $t_0>0$ be some number.
We apply successive approximations with
$X_t^{(0)}=0$, $X_t^{(1)}=f(t) \in L_1([0,t_0])$,
\begin{equation}\label{eq:approx}
X_t^{(n)} = \int_0^t u\left (X_s^{(n-1)}\right )\,ds + f(t) \in L_1([0,t_0]).
\end{equation}
Then for any $0<t\le t_0$,
\begin{align*}
\MoveEqLeft
\int_0^t \abs{X_s^{(n)} - X_s^{(n-1)}}ds
\le \int_0^t\!\!\int_0^s  \abs{u\left(X_v^{(n-1)}\right) - u\left(X_v^{(n-2)}\right)}dv\,ds
\\
&\le C \int_0^t\abs{X_v^{(n-1)} - X_v^{(n-2)}}(t-v)\,dv
\le\dots\le C^{n-1} \int_0^t \abs{f(s)} \frac{(t-s)^{n-1}}{(n-1)!}\,ds
\\
&\le \frac{(Ct)^{n-1}}{(n-1)!} \int_0^t \abs{f(s)}\,ds.
\end{align*}
This means that $X^{(n)}$ is a Cauchy sequence in $L_1([0,t_0])$, therefore there exists a limit $X_t = \lim_{n\to\infty} X_t^{(n)}$ in $L_1([0,t_0])$.
It is clear that $X$ is a solution of \eqref{eq:deterministic}.
Uniqueness follows from the Gronwall inequality.

Now let us consider the case  when \ref{(D2)} holds.
As before, let $t_0>0$ be fixed, and $f(t) \le C = C(t_0)$.
With $X_t^{(0)} = 0$, $X_t^{(1)} = f(t)$ is locally bounded, and every $X^{(n)}$ that is defined by \eqref{eq:approx} is locally bounded as well.
Moreover,
\begin{align*}
\abs{X_t^{(n)}} &\le \abs{f(t)} + Ct + C \int_0^t \abs{X_s^{(n-1)}} ds
\\
&\le \abs{f(t)} + Ct + C\int_0^t (\abs{f(s)} +Cs) ds + C^2 \int_0^t \abs{X_s^{(n-2)}} (t-s)\, ds
\le\dots\le
\\
&\le \abs{f(t)} + Ct + C\int_0^t (\abs{f(s)} + Cs) e^{C(t-s)} ds,
\end{align*}
therefore, $X^{(n)}$ are totally locally bounded.
Existence of the limit that is a unique solution of \eqref{eq:deterministic} is evident.
\end{proof}

\subsection{Stochastic equation}

Now let us return to the equation \eqref{main.object}, that is, let us consider the Volterra--L\'evy process $Y_t = \int_0^t g(t,s)\,dZ_s$ instead of the deterministic function $f$.
According to Lemma~\ref{l:ex-un_determ}, in order to obtain the existence and uniqueness of a solution, it suffices to establish either local integrability or local boundedness of $Y$.

First, we study the sufficient conditions for integrability.
Namely, we present the conditions supplying
$\ex\int_0^T\abs{Y_t}\,dt<\infty$.
If the assumption \ref{(A1)} holds, then by \eqref{eq:est1},
\[
\ex\int_0^T\abs{Y_t}\,dt
\le \int_0^T\!\left(\ex\abs{Y_t}^p\right)^\frac1p\,dt
\le C\left(\int_{\R}|x|^p\pi(dx)\right)^{\frac1p} \int_0^T\norm{g(t,\cdot)}_{L_p([0,t])}\,dt,
\]
therefore, the sufficient condition for integrability is
$\int_0^T\norm{g(t,\cdot)}_{L_p([0,t])}\,dt<\infty$.

Similarly, if the assumption \ref{(A2)} holds, then using \eqref{eq:est2} we get
\begin{multline*}
\ex\int_0^T\abs{Y_t}\,dt
\le C a^{\frac12} \int_0^T\norm{g(t,\cdot)}_{L_2([0,t])}\,dt
\\
+ C\left(\int_{\R}|x|^p\pi(dx)\right)^{\frac1p} \int_0^T\norm{g(t,\cdot)}_{L_p([0,t])}\,dt.
\end{multline*}
Since $p\ge2$, we see that again the sufficient condition for integrability has the form
$\int_0^T\norm{g(t,\cdot)}_{L_p([0,t])}\,dt<\infty$.
In the Gaussian case the second term vanishes, hence a weaker condition is required, namely $\int_0^T\norm{g(t,\cdot)}_{L_2([0,t])}\,dt<\infty$.

Now let the kernel $g$ satisfy the assumption \ref{(B1)}.
Then
\begin{align*}
\int_0^T\norm{g(t,\cdot)}_{L_p([0,t])}\,dt
&\le C \int_0^T t^\alpha\left(\int_0^t s^{\beta p} (t-s)^{\gamma p}\,ds\right)^{\frac 1p} \,dt
\\
&\le C \int_0^T t^{\alpha+\beta+\gamma +\frac 1p} \,dt,
\end{align*}
where we have used the equality
$\int_0^t s^{\beta p} (t-s)^{\gamma p}\,ds = \Beta(\beta p+1,\gamma p+1) t^{\beta p + \gamma p + 1}$ (assuming that $\beta>-\frac1p$, $\gamma>-\frac1p$).
Consequently, under the assumption \ref{(B1)} the condition $\int_0^T\norm{g(t,\cdot)}_{L_p([0,t])}\,dt<\infty$ holds, if $\alpha+\beta+\gamma +\frac 1p>-1$.

Similarly to Lemmas \ref{l:pure-jump}--\ref{l:general}, we can consider three cases. Thus, we arrive at the following result.
\begin{theorem}
Assume that one of the following assumptions holds:
\begin{enumerate}
\item
$p\ge1$, $a=0$, $\int_\R\abs{x}^p\,\pi(dx)<\infty$,
the condition \ref{(B1)} holds with some $\alpha\in\R$, $\beta>-\frac1p$, $\gamma>-\frac1p$ such that $\alpha+\beta+\gamma>-\frac1p-1$;
\item
$p\ge2$, $\int_\R\abs{x}^p\,\pi(dx)<\infty$,
the condition \ref{(B1)} holds with some $\alpha\in\R$, $\beta>-\frac1p$, $\gamma>-\frac1p$ such that $\alpha+\beta+\gamma>-\frac1p-1$;
\item
$Z$ is a Brownian motion,
the condition \ref{(B1)} holds with $p=2$, $\alpha\in\R$, $\beta>-\frac12$, $\gamma>-\frac12$ such that $\alpha+\beta+\gamma>-\frac32$.
\end{enumerate}
Then $\ex\int_0^T\abs{Y_t}\,dt<\infty$.
Consequently, if the coefficient $u$ satisfies the assumption \ref{(D1)} 1) of Lemma~\ref{l:ex-un_determ}, then the equation \eqref{main.object} has a unique solution.
\end{theorem}

Now we adapt the condition \ref{(D2)} 3) of Lemma~\ref{l:ex-un_determ} to the stochastic case.
Since continuity is a sufficient condition for local boundedness, we obtain the following corollary from Lemmas \ref{l:pure-jump}--\ref{l:general}.

\begin{theorem}
Assume that one of the following assumptions holds:
\begin{enumerate}
\item
$p\ge1$,
$a=0$, $\int_\R\abs{x}^p\,\pi(dx)<\infty$,
the conditions \ref{(B1)} and \ref{(B2)} hold with some $\alpha\in\R$, $\beta>-\frac1p$, $\gamma>-\frac1p$, $\delta>\frac1p$ such that $\alpha+\beta+\gamma>-\frac1p$, $\kappa>0$;
\item
$p\ge2$ we have
$\int_\R\abs{x}^p\,\pi(dx)<\infty$ and
the conditions \ref{(B1)} and \ref{(B2)} hold with some $\alpha\in\R$, $\beta>-\frac1p$, $\gamma>-\frac1p$, $\delta>\frac1p$ such that $\alpha+\beta+\gamma>-\frac1p$, $\kappa>0$;
\item
$Z$ is a Brownian motion,
the conditions \ref{(B1)} and \ref{(B2)} hold with $p=2$, $\alpha\in\R$, $\beta>-\frac12$, $\gamma>-\frac12$, $\delta>0$ such that $\alpha+\beta+\gamma>-\frac12$, $\kappa>-\frac12$.
\end{enumerate}
Then $Y$ has a.\,s. continuous (hence, locally bounded) sample paths.
Consequently, if the coefficient $u$ satisfies the assumptions \ref{(D2)} 1), 2) of Lemma~\ref{l:ex-un_determ}, then the equation \eqref{main.object} has a unique solution.
\end{theorem}

We remark that it seems that there no general results about solutions of stochastic differential equations \eqref{main.object} with Volterra--L\'evy noise without some form of Lipschitz continuity assumptions.
There are instead some papers dealing with some classes of such equations also with exploding drift.
We refer e.\,g.\ to \cite{New2} for a short survey and the study of a class of such equations.

In the next section we address another class of equation without Lipschitz drift.
We focus on Volterra--Gaussian processes.
The particular case of fractional Brownian motion was considered in \cite{NO}.

\section{Equations with Volterra--Gaussian processes}
\label{sec:5}

Now our goal is to consider equations with additive noise represented by various Volterra--Gaussian processes, some of which  were introduced in \cite{MSS}. Our aim is to relax the conditions on the drift coefficient, in a similar fashion to what was done in the paper \cite{NO}.
Remark that, in \cite{NO}, the noise was fractional Brownian motion, but here we deal with more general noise.

\subsection{Girsanov theorem. Definition of weak and strong solutions}

Let $\set{\F_t^V,t\in[0,T]}$ denote the natural filtration of $V$, where $V$ can be either $Y$ defined by \eqref{eq:gfbm} and \eqref{eq:kernK}, or it can be $\widehat{Y}$ is defined by \eqref{eq:gfbm1} and \eqref{eq:genker}.
For some process $u=\set{u_t,t\in[0,T]}$ with integrable trajectories, denote
\[
z(s) = \left(a^{-1} D^h_{0+} \left (u b^{-1}\right)\right)(s),
\quad
\hat{z}(s) = \left(\hat a^{-1} I_{0+}^{\hat{h}}\left(\hat b^{-1} u\right)\right)(s).
\]
Let
\[
\xi_T = \exp\set{-\int_0^T z(s)\,dW_s
-\frac12 \int_0^T z^2(s)\,ds},
\]
and
\[
\widehat{\xi}_T = \exp\set{-\int_0^T \hat{z}(s)\,dW_s
-\frac12 \int_0^T z^2(s)\,ds},
\]
respectively.
\begin{theorem}\label{th1}
\begin{itemize}
\item[$1)$]
Let the assumptions \ref{(C1)}--\ref{(C3)} hold, and
let $u=\set{u_t,t\in[0,T]}$ be a   $F^Y$-adapted process with integrable trajectories.
Consider the transformation
\begin{equation}\label{eq:gfbm-trans}
V_0(t) = Y_t + \int_0^t u_s\,ds.
\end{equation}
Assume that
\begin{enumerate}[\it(i)]
\item
$z \in L_2([0,T])$ a.\,s.,
and
\item
$\ex\xi_T=1$.
\end{enumerate}
Then $V_0$ can be represented as
\[
V_0(t) = \int_0^t K(t,s) \,d B_s, \quad t\in[0,T],
\]
where $B$ is a $\F^Y$-Wiener process under the new probability $  \pr_B$ defined by $d  \pr_B / d\pr = \xi_T$.
\item[$2)$]
Let the assumptions \ref{(hatC1)}--\ref{(hatC3)} hold, and
let $u=\set{u_t,t\in[0,T]}$ be a   $F^{\widehat{Y}}$-adapted process with integrable trajectories.
Consider the transformation
\[
\widehat{V}_0(t) = \widehat{Y}_t + \int_0^t u_s\,ds.
\]
Assume that
\begin{enumerate}[resume*]
\item\label{th1-(iii)}
$\hat{z} \in L_2([0,T])$ a.\,s.,
and
\item\label{th1-(iv)}
$\ex\widehat{\xi}_T=1$
\end{enumerate}
Then $\widehat V_0$ can be represented as
\[
\widehat{V}_0(t) = \int_0^t \widehat{K}(t,s) \,d \widehat{B}_s, \quad t\in[0,T],
\]
where $\widehat{B}$ is a $\F^{\widehat{Y}}$-Wiener process under the new probability $  \pr_{\widehat{B}}$ defined by $d\pr_{\widehat{B}} / d\pr = \widehat{\xi}_T$.
\end{itemize}
\end{theorem}

\begin{proof} Let us prove only $1)$ since both statements are proved similarly.
Inserting \eqref{eq:gfbm} into \eqref{eq:gfbm-trans},
we can write
\[
V_0(t) = \int_0^t K(t,s) \,dW_s + \int_0^t u_s\,ds
=\int_0^t K(t,s) \,dB_s,
\]
where
\[
B_t
= W_t + \int_0^t \K^{-1}\left(\int_0^\cdot u_s\,ds\right) (r)\,dr.
\]
Using   \eqref{eq:sol}, we get
\[
B_t= W_t + \int_0^t a^{-1}(r) D^h_{0+} \left (ub^{-1}\right) (r)\,dr.
\]
Finally, by the standard Girsanov theorem, $B$ is a $\F^Y$-Wiener process under the probability $  \pr_B$.
\end{proof}

In the sequel, we study two stochastic differential equations
\begin{equation}\label{eq:sde-gvp}
X_t = x + V_t + \int_0^t u(s,X_s)\,ds,\quad t\in[0,T],
\end{equation}
where $x\in\R$, $u\colon[0,T]\times\R\to \R$
is a measurable function, $V = Y,\widehat{Y}$, where $Y$ is defined by \eqref{eq:gfbm} and \eqref{eq:kernK}, while $\widehat{Y}$ is defined by \eqref{eq:gfbm1} and \eqref{eq:genker}.
We shall consider both strong and weak solutions according to the definition below.

\begin{definition}
\begin{enumerate}[(i)]
\item
By a weak solution of equation \eqref{eq:sde-gvp} we mean a couple of processes $(V,X)$ on the filtered probability space $(\Omega,\F,\FF^V,\pr)$, such that
\begin{equation}\label{eq:w-sol}
V_t = \int_0^t K(t,s)\,dW_s
\quad\text{or}\quad
V_t = \int_0^t \widehat K(t,s)\,dW_s,
\end{equation}
respectively, with some Wiener process $W$, and $(V,X)$ satisfy \eqref{eq:sde-gvp}.

\item
By a strong solution of equation \eqref{eq:sde-gvp} we understand a process $X$ on $(\Omega,\F,\FF^V,\pr)$, and $V$ is of the form \eqref{eq:w-sol} with the fixed Wiener process $W$.
\end{enumerate}
\end{definition}

\subsection{Weak existence and weak uniqueness}

Let the coefficients $a$, $b$, $c$ satisfy the assumptions \ref{(C1)}--\ref{(C4)}.
Then, according to Proposition~\ref{prop:holder}, the stochastic process $Y$ has a modification satisfying H\"{o}lder condition up to order $\nu\in(0,1/2).$
\begin{theorem}\label{th2}
\begin{enumerate}[(i)]
\item Assume that $u(s,x)$ satisfies the sublinear growth condition: there exist  such $0<\alpha<1$ and $C>0$ that
\begin{equation}\label{eq:lin-grow}
\abs{u(t,x)} \le C(1+\abs{x}^\alpha),
\end{equation}
and   H\"{o}lder condition in space and time: there exist    $0<\beta\le 1, 0<\gamma< 1$ and $C>0$ such that for any $s,t\in [0,T]$ and any $x,y\in \mathbb{R}$
$$|u(t,x)-u(s,y)|\leq C\left(\abs{t-s}^\beta+\abs{y-x}^\gamma\right).$$

Additionally to \ref{(C1)}--\ref{(C4)}, let also functions $a$, $b$ and $h$ satisfy the following assumption: there exist $C>0$ and $\nu'\in(0,\nu)$ such that
\begin{equation}\begin{gathered}\label{four}
\int_0^T  a^{-2} (s)   h^2 (s)    b^{-2}(s)\,ds \le C,\\
\int_0^Ta^{-2} (t)\left(\int_0^t \abs{  h' (t-r)} \abs{b^{-1}(t) - b^{-1}(r)}\,dr\right)^2\,dt \le C, \\
\int_0^Ta^{-2} (t) {b^{-2}(t)}\left(\int_0^t \abs{  h' (t-r)} (t-r)^{\beta}\,dr\right)^2\,dt \le C,\\
\int_0^Ta^{-2} (t) {b^{-2}(t)}\left(\int_0^t \abs{  h' (t-r)} (t-r)^{\gamma\nu'}\,dr\right)^2\,dt \le C.
\end{gathered}\end{equation}
Then the equation \eqref{eq:sde-gvp} with $V = Y$ has a unique weak solution.
\item Assume that $u(s,x)$ satisfies the sublinear growth condition \eqref{eq:lin-grow},
and, additionally to \ref{(hatC1)}--\ref{(hatC3)}, functions $\hat a$, $\hat b$ and $\hat h$ satisfy following  assumption: there exists  $C>0$ such that
\begin{equation}\label{eq:bounded}
\hat a^{-1} (s) \int_0^s \abs{\hat h (s-r)} \abs{\hat b^{-1}(r)}\,dr \le C.
\end{equation}
Then equation \eqref{eq:sde-gvp} with $V=\widehat Y$ has a unique weak solution.
\end{enumerate}
\end{theorem}

\begin{remark}
Let us check the conditions \eqref{four} in the case when $V$ is a fractional Brownian motion.

$(i)$ Let $H>\frac12$, $a(s) = s^{\frac12-H}$, $b(s) = s^{H-\frac12}$, $c(s) = s^{H-\frac32}$, $h(s) = s^{\frac12-H}$. Then
\[
\int_0^T a^{-2}(s) h^2(s) b^{-2}(s)\,ds
=\int_0^T s^{1-2H}\,ds = (2-2H)^{-1} T^{2-2H};
\]
\begin{align*}
\MoveEqLeft
\int_0^T a^{-2}(t)\left(\int_0^t\abs{h'(t-r)} \abs{b^{-1}(t) - b^{-1}(r)}\,dr\right)^2 dt
\\
&= C\int_0^T t^{2H-1}\left(\int_0^t (t-r)^{-\frac12-H} \left(r^{\frac12-H} - t^{\frac12-H}\right)\,dr\right)^2 dt
\\
&= C\int_0^T t^{2H-1}\cdot t^{-1-2H}\cdot t^{1-2H} t^2\,dt
\left(\int_0^1 (1-r)^{-\frac12-H} \left(r^{\frac12-H} - 1\right)\,dr\right)^2
\\
&= C T^{2-2H}\left(\int_0^1 (1-r)^{-\frac12-H} \left(r^{\frac12-H} - 1\right)\,dr\right)^2.
\end{align*}
Integral
$\int_0^1 (1-r)^{-\frac12-H} \left(r^{\frac12-H} - 1\right)\,dr$
is finite, since around zero
\[
(1-r)^{-\frac12-H} \left(r^{\frac12-H} - 1\right) \sim r^{\frac12-H} - 1
\]
and around 1
\[
(1-r)^{-\frac12-H} \left(r^{\frac12-H} - 1\right) \sim (1-r)^{\frac12-H}.
\]
Further,
\begin{align*}
\MoveEqLeft
\int_0^T \left(a^{-2}b^{-2}\right)(t) \left(\int_0^t\abs{h'(t-r)}(t-r)^\beta\,dr\right)^2 dt
\\
&= C \int_0^T \left(\int_0^t (t-r)^{-\frac12-H+\beta}\,dr\right)^2 dt
\le C
\end{align*}
if $-\frac12-H+\beta > -1$, or $\beta > H-\frac12$.
Finally,
\begin{align*}
\MoveEqLeft
\int_0^T \left(a^{-2}b^{-2}\right)(t) \left(\int_0^t\abs{h'(t-r)}(t-r)^{\gamma\nu'}\,dr\right)^2 dt
\\
&= C \int_0^T \left(\int_0^t (t-r)^{-\frac12-H+\gamma\nu'}\,dr\right)^2 dt
\le C
\end{align*}
if $-\frac12-H+\gamma\nu'$ or $\gamma\nu' > H-\frac12$.
But in this case $\nu'$ can be any number from 0 to $H$, therefore, condition
$\gamma\nu' > H-\frac12$ holds if $\gamma H > H-\frac12$, or
$\gamma > 1-\frac{1}{2H}$.
Therefore assumptions \eqref{four} hold for $\beta > H-\frac12$, $\gamma > 1-\frac{1}{2H}$.

$(ii)$
Let $H<\frac12$. Then $\hat a(s) = Cs^{\frac12-H}$, $\hat b(s) = \hat c(s) = s^{H-\frac12}$, $\hat h(s) = s^{-\frac12-H}$, therefore
\[
\hat a^{-1}(s)\int_0^s \abs{\hat h(s-r)} \abs{\hat b^{-1}(r)}\,dr
= C s^{H-\frac12} \int_0^s (s-r)^{-\frac12-H} r^{\frac12-H}\,dr
= C s^{\frac12-H} \le C,
\]
so \eqref{eq:bounded} holds.
\end{remark}

\begin{proof}
First, we give some upper bounds for $z(s)$ and $\hat z(s)$ in order to confirm that theorem's conditions supply Novikov conditions for $\xi_T$ and $\widehat \xi_T$, and therefore $\xi_T$ and $\widehat \xi_T$ satisfy Theorem \ref{th1}.
Then the proofs of $(i)$ and $(ii)$ are similar, therefore we continue   only with   the second statement,
  dividing the proof  into several steps and  refer to the paper   \cite{NO} for additional detail.

Concerning $z(s)$, by Lemma \ref{lem8} \ref{l1-iii}, we have that
\begin{align*}
z(s)
%&= a^{-1}(s) D^h_{0+} \left(u b^{-1}\right)(s)
&= \left( a^{-1} h b^{-1} \right)(s) u\left(s,Y_s+x\right)
\\
&\quad+ a^{-1}(s)\int_0^s \left(u\left(z,Y_z+x\right)b^{-1}(z)
- u\left(s,Y_s+x\right) b^{-1}(s)\right) h'(s-z)\,dz
\\
&= J_1(s) + J_2(s).
\end{align*}
Let us construct upper bounds for $J_1$ and $J_2$.
Namely, we are interested in two integrals.
First,
\begin{equation}\label{eq:ubJ_1}
\begin{split}
\int_0^T J_1^2(s)\,ds &\le C \left(1+\sup_{0\le s\le T} \abs{Y_s+x}^{2\alpha}\right) \int_0^T \left(a^{-2} h^2 b^{-2}\right)(s)\,ds
\\
&\le C\left(1+\sup_{0\le s \le T}\abs{Y_s}^{2\alpha}\right),
\end{split}
\end{equation}
according to 1st assumption in \eqref{four}.

Second,
\begin{align*}
\MoveEqLeft
\int_0^T J_2^2(s)\,ds
\le C \left(1+\sup_{0\le s\le T} \abs{Y_s+x}^{2\alpha}\right)
\\*
&\quad\quad\times \int_0^T a^{-2}(s) \left(\int_0^s\abs{b^{-1}(z) - b^{-1}(s)}
\abs{h'(s-z)}\,dz\right)^2 ds
\\
&+ \int_0^T \left(a^{-2}b^{-2}\right)(s) \left(\int_0^s \abs{u(s,Y_s+x) - u(z,Y_z+x)}
\abs{h'(s-z)}\,dz\right)^2 ds
\\
&= M_1 + M_2.
\end{align*}
Obviously,
\begin{equation}\label{eq:ubM_1}
M_1 \le C \left(1+\sup_{0\le s\le T}\abs{Y_s}^{2\alpha}\right),
\end{equation}
according to the 2nd assumption in \eqref{four}.
Concerning $M_2$, it admits the following upper bound:
\begin{align*}
M_2 &\le C \int_0^T (ab)^{-2}(s)\left(\int_0^s(s-z)^\beta\abs{h'(s-z)}\,dz\right)^2 ds
\\
&\quad+ C \int_0^T (ab)^{-2}(s)\left(\int_0^s\abs{Y_s-Y_z}^\gamma \abs{h'(s-z)}\,dz\right)^2 ds
\\
&= N_1 + N_2.
\end{align*}
According to 3rd assumption in \eqref{four}, $N_1\le C$.
Further, due to the 4th assumption from \eqref{four},
\begin{align*}
N_2 &\le C\left(\sup_{0\le s< t\le T} \frac{\abs{Y_s-Y_t}^\gamma}{(t-s)^{\gamma \nu'}}\right)^2 \int_0^T (ab)^{-2}(s)
\left(\int_0^s (s-z)^{\gamma\nu'}\abs{h'(s-z)}\,dz\right)^2ds
\\
&\le C\left(\sup_{0\le s< t\le T} \frac{\abs{Y_s-Y_t}}{(t-s)^{\nu'}}\right)^{2\gamma} \eqqcolon C G,
\end{align*}
and due to the fact that $2\gamma<1$ and to \cite{Fernique},
$\ex\exp CG<\infty$ for any $G>0$.
Combining this with \eqref{eq:ubJ_1} and \eqref{eq:ubM_1},
we conclude that
$\ex\exp\set{\frac12\int_0^T z^2_s\,ds}<\infty$,
and Novikov condition holds, consequently, $\widehat\xi_T$ is indeed a density function.

Concerning $\hat z(s)$, let us provide the following calculations:
\[
\hat z(s) = \hat a^{-1}(s) \int_0^s \hat h(s-r)\, \hat b^{-1}(r)
\,u\left(r,\widehat Y_r +x\right)\,dr
\]
and, according to \eqref{eq:lin-grow},
\begin{equation}\label{growth}\begin{gathered}
\abs{\hat z(s)}^2  \le C \hat a^{-2}(s) \left(\int_0^s \abs{\hat h(s-r)}\,
\abs{\hat b^{-1}(r)}\left(1+ \bigl\lvert\widehat Y_r+x\bigr\rvert^\alpha\right)\,dr\right)^2
\\
 \le C\left(1+ \sup_{r\in[0,T]}\bigl\lvert\widehat Y_r\bigr\rvert^{2\alpha}\right) \hat a^{-2}(s) \left(\int_0^s \abs{\hat h(s-r)}\, \abs{\hat b^{-1}(r)}\,dr\right)^2.
\end{gathered}\end{equation}
Under assumption \eqref{eq:bounded}
\[
\abs{\hat z(s)}^2 \le C\left(1+ \sup_{r\in[0,T]}\bigl\lvert\widehat Y_r\bigr\rvert^{2\alpha}\right).
\]
Then it follows from the fact that $2\alpha<2$ and integrability of supremum of Gaussian process \cite{Fernique}
that
$\sup_{0\le s \le T} \ex\exp\set{\rho\sup_{0\le s \le T}\abs{\hat z(s)}^2}<\infty$
for any $\rho>0$, and this inequality supplies Novikov condition for $\widehat\xi_T $.

Now we continue with the proof of $(ii)$.
We consider the two cases of $V$.

\textbf{(a)}
Together with Theorem~\ref{th1}, we can conclude  that $\widetilde Y$ is a Volterra--Gaussian process of the form $\widetilde Y_t = \int_0^t \widehat K (t,s)\,d\widetilde B_s$,
where $\widetilde B$ is a Wiener process with respect to the probability measure
$\pr_{\widetilde B}$ defined by
$d\pr_{\widetilde B} / d\pr = \widehat\xi_T$,
where
\[
\widehat\xi_T = \exp\set{\int_0^T\hat z(s)\,dW_s
- \frac12\int_0^T\hat z^2(s)\,ds}.
\]
It means that the couple $(\widetilde Y, \widehat Y+x)$ creates a weak solution of \eqref{eq:sde-gvp} with $V=\widehat Y.$

\textbf{(b)}
Now let us apply and modify the  approach from \cite{NO} concerning the proof of uniqueness in law and pathwise uniqueness of the equations with additive fractional noise.
Namely, consider any solution of the equation
\[
X_t = x + \int_0^t u(s,X_s)\,ds + \widehat Y_t,
\]
where
$\widehat Y_t = \int_0^t \widehat K(t,s)\,dB_s$,
$B$ is some Wiener process, and define
\[
\hat z(s) = \hat a^{-1}(s) \int_0^s h(s-r) \hat b^{-1}(r) u(r,X_r)\,dr.
\]
Note that $X\in \mathbb{C}([0,T])$, therefore, due to sublinear growth condition,
\[
\sup_{0\le v\le r}\abs{u(v,X_v)}
\le C\left (1+\sup_{0\le v\le r}\abs{X_v}^{2\alpha}\right ) < \infty
\quad\text{a.\,s.}
\]
Also,
$\sup_{0\le t\le T}\bigl\lvert\widehat Y_t\bigr\rvert<\infty$ a.\,s.
Therefore, from Gronwall inequality we get that
\[
\sup_{0\le t\le T}\abs{X_t} \le \left(\abs{x} + \sup_{0\le t\le T}\bigl\lvert\widehat Y_t\bigr\rvert + CT \right) e^{CT},
\]
and in turn it implies that, similarly to \eqref{growth}, under assumption \eqref{eq:bounded}, for any $s\in[0,T]$
\[
\abs{\hat z(s)}^2 \le C\left(1+\sup_{0\le t\le T}\abs{X_t}^{2\alpha}\right)
\le C_1 \left(1 + \sup_{0\le t\le T}\bigl\lvert\widehat Y_t\bigr\rvert^{2\alpha}\right).
\]
It means that  w.\,r.\,t.\ the measure $\widehat{\pr}$ such that
\begin{equation}\label{problike}
\frac{d\widehat{\pr}_T}{d\pr_T} = \exp\set{-\int_0^T \hat z(s)dB_s
-\frac12 \int_0^T \hat z^2(s)\,ds},
\end{equation} $X_t-x$ has the same distribution as the process $\int_0^t \hat K(t,s)\,dV_s$, where $V$ is a Wiener process,
 $V_s=B_s+\int_0^s\hat{z}(u)du,$
and the right-hand side of \eqref{problike} indeed defines a probability measure.

Further, for any bounded measurable functional $\Phi$ on $\mathbb{C}([0,T])$,
\begin{equation}\label{funcsl}
\begin{split}
\MoveEqLeft
\ex_\pr\Phi(X-x)
= \int_\Omega \Phi (\xi-x) \frac{d\pr_T}{d\widehat\pr_T}(\xi)\,d\widehat\pr_T
\\
&= \ex_{\widehat\pr}\left(\Phi(X-x)\exp\set{\int_0^T\hat z(s)\,dB_s
+\frac12\int_0^T\hat z^2(s)\,ds}\right)
\\
&= \ex_{\widehat\pr}\left(\Phi(X-x)\exp\set{\int_0^T\hat a^{-1}(s)
\int_0^s h(s-r)\hat b^{-1}(r) u(r,X_r)\,dr\,dB_s
\right.\right.
\\
&\quad+\left.\left.\frac12\int_0^T\left(\hat a^{-1}(s)
\int_0^s h(s-r)\hat b^{-1}(r) u(r,X_r)\,dr\right)^2 ds}\right)
\\
&= \ex_{\widehat\pr}\left(\Phi(X-x)\exp\set{\int_0^T\hat a^{-1}(s)
\int_0^s h(s-r)\hat b^{-1}(r) u(r,X_r)\,dr\,dV_s
\right.\right.
\\
&\quad-\left.\left.\frac12\int_0^T\left(\hat a^{-1}(s)
\int_0^s h(s-r)\hat b^{-1}(r) u(r,X_r)\,dr\right)^2ds}\right)
\\
&= \ex_{\pr}\Phi\left(\int_0^\cdot \widehat K(\cdot,s)\,dB_s\right)
\\
&\quad\times
\exp\set{\int_0^T\hat a^{-1}(s)
\int_0^s h(s-r)\hat b^{-1}(r)\right.
 u\left(r,x + \int_0^T\widehat K(r,z)\,dB_z\right)\,dr\,dB_s
\\
&\quad-\left.\frac12\int_0^T\left(\hat a^{-1}(s)
\int_0^s h(s-r)\hat b^{-1}(r) u\left(r,x + \int_0^T\widehat K(r,z)\,dB_z\right)\,dr\right)^2 ds}
\\
&= \ex_{\pr}\Phi\left(\int_0^\cdot \widehat K(\cdot,s)\,dV_s\right).
\end{split}\end{equation}
Taking \eqref{funcsl} into account, we conclude that any two weak solutions have the same distribution, so we established weak uniqueness.
\end{proof}

\subsection{Pathwise uniqueness of two weak solutions}
 Now we consider only  equation
 \begin{equation}\label{eq:sde-gvp_Y}
X_t = x + Y_t + \int_0^t u(s,X_s)\,ds,\quad t\in[0,T],
\end{equation}
where $x\in\R$, $u\colon[0,T]\times\R\to \R$
is a measurable function,   $Y$ is defined by \eqref{eq:gfbm} and \eqref{eq:kernK}.
\begin{theorem}\label{th3}
Let coefficients $a$, $b$, $c$ satisfy assumptions \ref{(C1)}--\ref{(C3)} and \ref{(C5)}.
Let also coefficient $u(s,x)$ satisfy conditions of item $(i)$, Theorem \ref{th2}.
Then any two weak solutions of equation \eqref{eq:sde-gvp_Y} with the same Wiener process $W$ participating in the representation of $Y$, coincide a.\,s.
\end{theorem}

\begin{proof}
According to Proposition~\ref{prop:holder}, the condition \ref{(C5)} supplies that the process $Y$ on any interval $[t_1+t_2,T]$ has a modification that satisfies
H\"older condition up to order
$\mu = \frac32 - \frac{1}{q_1} - \max\left(\frac12,\frac1p+\frac{1}{r_1},\frac{1}{p_1}+\frac{1}{r}\right)>\frac12$.
So, consider any $0 < \eps < T$, and on the interval $[\eps,T]$ apply It\^o formula to the process $\max\left(X^1_t,X^2_t\right)$, where $X^1$ and $X^2$ are two weak solutions with the same Wiener process $W$.
Observing that $X^1$ and $X^2$ are H\"older up to order $\mu>\frac12$ on $[\eps, T]$, which implies that the quadratic variation of $X^1-X^2$ is zero, we get that for any $t\in[\eps,T]$
\begin{align*}
\MoveEqLeft
\max\left(X_t^1,X_t^2\right) - \max\left(X_\eps^1,X_\eps^2\right)
= X_t^1 - X_\eps^1 +\left(X_t^2-X_t^1\right)_+ - \left(X_\eps^2-X_\eps^1\right)_+
\\
&= Y_t - Y_\eps + \int_\eps^t u\left (s,X_s^1\right )\,ds
+ \int_\eps^t \left(u\left (s,X_s^2\right ) - u\left (s,X_s^1\right )\right) \ind\set{X_s^2>X_s^1}\,ds
\\
&= Y_t - Y_\eps + \int_\eps^t u\left(s,\max\left(X_s^1,X_s^2\right)\right)\,ds.
\end{align*}

Let $\eps\to0$.
Then it follows from continuity of $Y$ and $u$ that $Y_\eps\to0$ a.\,s., and
\[
\int_\eps^t u\left(s,\max\left(X_s^1,X_s^2\right)\right)\,ds
\to \int_0^t u\left(s,\max\left(X_s^1,X_s^2\right)\right)\,ds
\quad\text{a.\,s.}
\]
Moreover,
$\max\left(X_\eps^1,X_\eps^2\right) \to x$ a.\,s.

Finally,
\[
\max\left(X_t^1,X_t^2\right) = x + Y_t + \int_0^t u \left(s,\max\left(X_s^1,X_s^2\right)\right)\,ds.
\]
It means that $\max\left(X_t^1,X_t^2\right)$ (and similarly
$\min\left(X_s^1,X_s^2\right)$) satisfies equation \eqref{eq:sde-gvp_Y}.
Due to the weak uniqueness proved in Theorem~\ref{th2},
$\max\left(X_t^1,X_t^2\right)$ and $\min\left(X_s^1,X_s^2\right)$
have the same distribution, whence $X_t^1 = X_t^2$ a.\,s., and from continuity of $X^1$ and $X^2$, $X_t^1 = X_t^2$, $t\in[0,T]$, a.\,s.
\end{proof}

\begin{remark}
1. Condition \ref{(C5)} is fulfilled in the case when $Y=B^H$ with $H>\frac12$.
In this case we can put
$p_1=q_1=r_1=\frac3\eps$,
where $0<\eps<\min\set{(H-\frac12),3(1-H),\frac12}$,
$\frac1p=H=\frac12+\frac\eps3$, $\frac1q=\frac\eps3$,
$\frac1r=\frac32-H+\frac\eps3$.
Then
\[
\mu = \frac32 - \frac\eps3 - \max\set{\frac12, \frac32-H+\frac{2\eps}{3}, H-\frac12+\frac{2\eps}{3}}
=H-\eps>\frac12.
\]

2. In the case when we cannot guarantee that $Y$ is H\"older up to some order $\mu>\frac12$ (for example, in the case when $Y=B^H$ with $H<\frac12$)
formula It\^o for $\max\left(X_t^1,X_t^2\right)$ has another form,
and the statement like Theorem \ref{th2} is an open problem.
\end{remark}

\subsection{Existence and uniqueness of strong solutions}

We conclude with a straightforward consequence of Theorems \ref{th2} and \ref{th3}.
\begin{theorem}
Under the assumptions of Theorem \ref{th3}, equation \eqref{eq:sde-gvp_Y} has a unique strong solution.
\end{theorem}

\section*{Appendix. Elements of fractional calculus for Sonine pairs}%\label{app:A}

Here we consider some notions similar to the notions of the fractional integral and of the fractional derivative proper to classical fractional calculus.
\begin{definition}\label{def:fracintfracder}
Let functions $c$ and $h$ from $L_1([0,T])$ create a Sonine pair.
Introduce the operators, similar to operators of fractional integral and fractional derivative:
\begin{gather*}
\left(I_{0+}^c f\right)(t) = \int_0^t c(t-s) f(s)\,ds,
\quad f\in L_1([0,T]),
\\
\left(D_{0+}^h f\right)(t) = \frac{d}{ds} \left(\int_0^t h(t-s) f(s) \, ds \right),
\end{gather*}
where $f\colon [0,T] \to \R$ is such that
\[
\int_0^t h(t-s) f(s) \,ds \in AC([0,T]).
\]
\end{definition}

Now we can here establish some properties of the operators
$I_{0+}^c$ and $D_{0+}^h$.
Denote
\[
I_{0+}^c \bigl(L_1([0,T])\bigr) = \set{\psi\colon[0,T]\to\R :
\psi(t) = \left(I_{0+}^c\varphi\right)(t),\, \varphi\in L_1([0,T])}.
\]

\begin{lemma}\label{lem8}
\begin{enumerate}[label=\textit{(\roman*)}]
\item\label{l1-i}
Let $f\in L_1([0,T])$. Then
\[
\left(D_{0+}^h I_{0+}^c f\right)(t) = f(t)
\quad\text{a.\,e.}
\]

\item\label{l1-ii}
Let $f\in I_{0+}^c \bigl(L_1([0,T])\bigr)$.
Then
\[
\left(I_{0+}^c D_{0+}^h f\right)(t) = f(t),
\quad t\in[0,T].
\]

\item\label{l1-iii}
Let $h\in C^1(0,T)$, there exist $\beta>0$ such that
$\lim_{s\to0} s^{\beta+1}h'(s) < \infty$.
Also, let $f$ be a H\"older function of order $\gamma$, and $\gamma>\beta$.
Then for any $t\in[0,T]$,
\[
\left(D_{0+}^h f\right)(t) = h(t)f(t) +
\int_0^t [f(z)-f(t)] h'(t-z)\,dz.
\]
\end{enumerate}
\end{lemma}

\begin{proof}
\ref{l1-i}
Obviously,
\begin{align*}
\left(D_{0+}^h I_{0+}^c f\right)(t)
&= \frac{d}{dt}\left(\int_0^t h(t-s)\left(\int_0^s c(s-u) f(u)\,du\right)ds\right)
\\
&= \frac{d}{dt}\left(\int_0^t f(u)\left(\int_u^t h(t-s) c(s-u)\,ds\right)du\right)
\\
&= \frac{d}{dt}\left(\int_0^t f(u)\,du\right)
= f(t)
\quad\text{a.\,e.}
\end{align*}

\ref{l1-ii}
Let $f(t)=\left(I_{0+}^c\varphi\right)(t)$, $\varphi\in L_1([0,T])$.
Then, according to \ref{l1-i},
\[
\left(I_{0+}^c D_{0+}^h f\right)(t)
=\left(I_{0+}^c D_{0+}^h I_{0+}^c\varphi\right)(t)
=\left(I_{0+}^c \varphi\right)(t)
= f(t),
\quad t\in[0,T].
\]

\ref{l1-iii}
Consider any $t\in(0,T)$ and $\Delta t>0$ (other values can be considered similarly).
Then
\begin{align*}
\Delta_f &\coloneqq \int_0^{t+\Delta t} h(t + \Delta t - s) f(s)\,ds
- \int_0^t h(t-s) f(s)\,ds
\\
&=\int_0^{t}\bigl(h(t + \Delta t - s) - h(t-s)\bigr) f(s)\,ds
+ \int_t^{t+\Delta t} h(t + \Delta t - s) f(s)\,ds
\\
&=\int_0^{t}\bigl(h(t + \Delta t - s) - h(t-s)\bigr) \bigl(f(s)-f(t)\bigr)ds
\\
&\quad+ \int_t^{t+\Delta t} h(t + \Delta t - s)\bigl(f(s)-f(t)\bigr)\,ds
+ f(t)\int_t^{t+\Delta t} h(s)\,ds.
\end{align*}
Evidently,
\[
\frac{1}{\Delta t}\left(f(t)\int_t^{t+\Delta t} h(s)\,ds\right)
\to f(t)h(t), \text{ a.\,e., as } \Delta t\to0.
\]
Furthermore,
\[
\frac{1}{\Delta t}\abs{\int_t^{t+\Delta t} h(t + \Delta t - s)
[f(s)-f(t)]ds}
=\abs{h(t + \Delta t - \theta_t)}\abs{f(\theta_t)-f(t)},
\]
where $\theta_t\in[t,t+\Delta t]$.
According to condition \ref{l1-iii} and L'H\^{o}pital's rule, for some constant $C>0$
\[
\lim_{\Delta t \to0}\abs{h(t + \Delta t - \theta_t)}\abs{f(\theta_t)-f(t)}
\le C\lim_{\Delta t \to0} \Delta t^{\gamma-\beta}=0.
\]
Finally, for $0<\varepsilon<t$
\begin{align*}
\MoveEqLeft
\abs{\int_0^{t}\left(\frac{h(t + \Delta t - s) - h(t-s)}{\Delta t}
-h'(t-s)\right) \bigl(f(s)-f(t)\bigr)ds}
\\
&= \abs{\int_0^{t}\bigl(h'(\theta_t-s) - h'(t-s)\bigr) \bigl(f(s)-f(t)\bigr)ds}
\\
&\le \abs{\int_0^{t-\eps}\bigl(h'(\theta_t-s) - h'(t-s)\bigr) \bigl(f(s)-f(t)\bigr)ds}
\\
&+ \abs{\int_{t-\eps}^{t}\bigl(h'(\theta_t-s) - h'(t-s)\bigr) \bigl(f(s)-f(t)\bigr)ds}
\\
&\le \abs{\int_0^{t-\eps}\bigl(h'(\theta_t-s) - h'(t-s)\bigr) \bigl(f(s)-f(t)\bigr)ds}
\\
&+  \int_{t-\eps}^{t} |h'(\theta_t-s)|  |f(s)-f(t)|ds
+\int_{t-\eps}^{t}\abs{h'(t-s)} \abs{f(s)-f(t)}\,ds.
\end{align*}
The first term, $\abs{\int_0^{t-\eps}\bigl(h'(\theta_t-s) - h'(t-s)\bigr) \bigl(f(s)-f(t)\bigr)ds}$, tends to 0 as $\Delta t\to0$ for any $\eps>0$.
Concerning the second term, it can be bounded as follows. For sufficiently small $\varepsilon$, it follows from $(iii)$ that
\[
\int_{t-\eps}^{t}\abs{h'(\theta_t-s)} \abs{f(s)-f(t)}\,ds
\le C \int_{t-\eps}^{t} (t-s)^{-1-\beta} (t-s)^\alpha\,ds
= C\eps^{\alpha-\beta},
\]
and, second,  the same is true for
$\int_{t-\eps}^{t}\abs{h'(t-s)} \abs{f(s)-f(t)}\,ds$,
and the proof follows.
\end{proof}

\section*{Acknowledgment}

The authors acknowledge that the present research is carried through within the frame and support of the ToppForsk project nr. 274410 of the Research Council of Norway with title STORM: Stochastics for Time-Space Risk Models.

\bibliographystyle{abbrv}
\bibliography{biblio}

\end{document}